\documentclass[10pt,a4paper,leqno]{amsart}

\title[Equality of the homogeneous and Gabor wave front set]{The equality of the homogeneous and the Gabor wave front set}

\author[R. Schulz]{Ren\'e Schulz}

\address{Leibniz Universit\"at Hannover, Institut f\"ur Analysis, Welfenplatz 1, D--30167 Hannover, Germany}

\email{rschulz@math.uni-hannover.de}

\author[P. Wahlberg]{Patrik Wahlberg}

\address{Department of Mathematics, Linn{\ae}us University, SE--351 95 V\"axj\"o, Sweden}

\email{patrik.wahlberg@lnu.se}

\keywords{}

\usepackage{tikz,xifthen}

\usepackage{latexsym}
\usepackage{amsmath}
\usepackage{amssymb}
\usepackage{amsthm}
\usepackage{amsfonts}
\usepackage{mathrsfs}
\usepackage{calc}
\usepackage{cite}
\usepackage{color}
\usepackage{graphicx}
\usepackage{enumerate}

\numberwithin{equation}{section}          

\newtheorem{thm}{Theorem}
\numberwithin{thm}{section}

\newcommand{\rubrik}{}
\newtheorem{prop}[thm]{Proposition}

\newtheorem{lem}[thm]{Lemma}
\newtheorem{mainlem}{Main Lemma}

\theoremstyle{definition}

\newtheorem{defn}[thm]{Definition}

\theoremstyle{remark}

\newcommand{\pd}[1] {\partial ^#1}
\newcommand{\pdd}[2] {\partial_{#1} ^{#2}}

\newcommand{\ro}{\mathbb R}
\newcommand{\no}{\mathbb N}
\newcommand{\rr}[1]{\mathbb R^{#1}}
\newcommand{\nn}[1]{\mathbb N^{#1}}

\newcommand{\zz}[1]{\mathbb Z^{#1}}

\newcommand{\co}{\mathbb C}
\newcommand{\cc}[1]{\mathbb C^{#1}}
\newcommand{\dd}{\mathrm {d}}

\newcommand{\HWF}{\mathrm{HWF}}

\newcommand{\ep}{\varepsilon}
\newcommand{\fy}{\varphi}

\newcommand{\supp}{\operatorname{supp}}

\newcommand{\wpr}{{\text{\footnotesize $\#$}}}

\newcommand{\eabs}[1]{\langle #1\rangle}

\newcommand{\charac}{\operatorname{char}}
\newcommand{\conesupp}{\operatorname{conesupp}}
\newcommand{\dbar}{{{{\ \mathchar'26\mkern-12mu \mathrm d}}}}

\newcommand{\WF}{\mathrm{WF}_G}

\newcommand{\cS}{\mathscr{S}}

\newcommand{\cF}{\mathscr{F}}
\newcommand{\cO}{\mathcal{O}}

\newcommand{\wt}{\widetilde}
\newcommand{\wh}{\widehat}

\newcommand{\im}{\rm Im \ }
\def\la{\langle}
\def\ra{\rangle}

\begin{document}

\begin{abstract}
We prove that H\"ormander's  global wave front set and Nakamura's homogeneous wave front set of a tempered distribution coincide. In addition we construct a tempered distribution with a given wave front set, and we develop a pseudodifferential calculus adapted to Nakamura's homogeneous wave front set.
\end{abstract}

\keywords{Microlocal analysis, homogeneous wave front set, Gabor wave front set, semiclassical 
analysis, Shubin calculus.
MSC 2010 codes: 35S05, 35A18, 35A22, 35A27, 42B37.}

\maketitle

\section{Introduction}

In this paper we prove the equality of the homogeneous wave front set, introduced by Nakamura \cite{Nakamura1},
and the global (Gabor) wave front set introduced by H\"ormander \cite{Hormander1}, of a tempered distribution on $\rr d$.
The homogeneous and the global wave front sets are both closed conical subsets of $T^*(\rr d) \setminus \{ 0 \}$ designed to encode global regularity.
The conical property here refers to the variables and covariables simultaneously, and not as is usual in microlocal analysis with respect to the covariables only, for fixed variables.
The global regularity means that the wave front set of $u \in \cS'(\rr d)$ is empty if and only if $u \in \cS(\rr d)$.
Thus the wave front set gives information on both smoothness and decay.

Such wave front sets arise in the study of propagation of singularities under partial differential equations.
The classical theory thereof, i.e. propagation of singularities on compact manifolds without boundary, or $\rr d$, is suitably covered by the classical theory of H\"ormander, see \cite{Hormander0}.

If in addition to lack of smoothness, growth singularities on unbounded spaces or compact spaces with boundaries are considered, the problem is not as well understood.
We give a brief exposition, not intended to be in any way a complete survey of the topic.

In one approach to address singularities at infinity, Melrose introduced the scattering wave front set, cf. \cite{Melrose1,Melrose2}.
For Euclidean spaces this coincides with the $\cS$ wave front set in \cite{Coriasco2} (cf. also \cite{Cordes}).
This theory was used to study tempered distributions and was shown to be adapted to the $SG$-calculus of pseudodifferential operators.
The notion extends the classical wave front set by components at infinity, defining a wave front set contained in $\rr d\times S_{d-1} \cup S_{d-1}\times S_{d-1} \cup S_{d-1}\times \rr d$ as a closed subspace, such that the classical wave front set is the first component.
These notions have been complemented by \cite{Wunsch1}, where the quadratic scattering wave front set was similarly introduced to study propagation of singularities, partly arising via quadratic oscillations, in Schr\"odinger-type equations on scattering manifolds.
We note that subsequently a family of global wave front sets has been introduced in \cite{Melrose3}.

On a different note Nakamura \cite{Nakamura1} introduced the homogeneous wave front set in order to study propagation of singularities under Schr\"odinger operators via methods typically used in semiclassical analysis.
In particular the symbols are compactly supported, smooth and dilated by a small parameter.
The homogeneous wave front set encodes global regularity and is quite different from the classical wave front set. However, there are some inclusion results (see \cite{Nakamura1}) and
Ito \cite{Ito1} has shown that the homogeneous wave front set is closely related to the quadratic scattering wave front set. This connection was used to rededuce some of the results in \cite{Wunsch1} with a different approach.

There exist versions of the homogeneous wave front set adapted to analytic \cite{Martinez2} and Gevrey \cite{Mizuhara1} regularity,
defined in terms of the short-time Fourier transform and decay estimates.
They are used to obtain similar results to those in \cite{Nakamura1} in their respective functional setting.
These techniques may be applied to the temperate case as well and, by the results of this paper, they can be understood to yield the same objects as
in \cite{Nakamura1}.

There exists another global wave front set suitable for the analysis of tempered distributions on $\rr d$, introduced by H\"ormander \cite{Hormander1}, which can be employed to study propagation of singularities under quadratic hyperbolic operators. It is defined by means of the Weyl calculus with Shubin symbols.
This wave front set has been shown to be characterizable in terms of the short-time Fourier transform and by means of Gabor frames, see \cite{Hormander1,Rodino1}. We use the nomenclature of \cite{Rodino1} and call this notion the Gabor wave front set.

This illustrates that quite a few notions of wave front sets encoding global regularity exist, however all of them are somehow connected.
Our contribution consists of a proof that two of them are equal, the Gabor wave front set introduced by H\"ormander and the homogeneous wave front set of Nakamura.

As a tool for the deduction of our main result, we develop a calculus of Shubin symbols that are dilated in the phase space variables by a small parameter (see Appendix \ref{appendix}). This calculus is inspired by semiclassical pseudodifferential calculi (cf. \cite{Martinez1,Zworski1}) where the dilation occurs only in the covariables, not in the variables.
We call this calculus a global semiclassical calculus, where the connection to semiclassical analysis is the one just described, and the term global is used to indicate that the framework is tempered distributions and Shubin symbols.

It is interesting to note that there are close connections between semiclassical notions of wave front sets (cf. \cite{Martinez1,Zworski1}) and H\"ormander's classical wave front set.
Both can be characterized in many ways, in particular via pseudodifferential operators and integral transforms.
Our main result may be interpreted as a version of the connection between classical and semiclassical wave front sets for wave front sets encoding global regularity.

The paper is organized as follows: in Section \ref{sec:prelim} we introduce our notation and recall the notions needed from pseudodifferential calculus and harmonic analysis. Section \ref{sec:wfintro} is devoted to the definition and the basic properties of the Gabor and the homogeneous wave front sets, respectively. We also introduce a parameter-dependent version of the short-time Fourier transform, which is used to relate the two wave front sets.
In Section \ref{sec:prep} we prepare for the proof of our main result. We prove estimates for pseudodifferential operators with dilated symbols and assigned support properties acting on a time-frequency translated Gaussian.

In Section \ref{sec:main} we use the preparations to show our main result, the equality of the homogeneous and the Gabor wave front sets.
In Section \ref{sec:exist}, as an excursus and complement to the theory, we construct explicitly, for any given closed conic subset $\Gamma \subseteq T^*(\rr d) \setminus \{0\}$, a tempered distribution with Gabor wave front set equal to $\Gamma$.
In Appendix \ref{appendix} we give an account of a pseudodifferential calculus adapted to symbols dilated in both variables and covariables simultaneously.
This calculus is a tool for the study of the homogeneous wave front set. In particular we show an invariance property of the homogeneous wave front set, which is needed in the proof of our main result, and we hope that this calculus may be useful in other situations.

\section{Preliminaries}
\label{sec:prelim}

An open ball of radius $r>0$ and center $(x_0,\xi_0) \in \rr {2d}$ is denoted $B_r(x_0,\xi_0)$ and $B_r = B_r(0)$.
The unit sphere in $\rr d$ is denoted $S_{d-1}$.
Positive constants will be denoted by $C$, sometimes with appropriate subscripts, and these constants may change value over inequalities.
We write $f (x) \lesssim g (x)$ provided there exists $C>0$ such that $f (x) \leq C g(x)$ for all $x$ in the domain of $f$ and $g$.
If $f(h) = \cO(h^N)$ for $h \in (0,1]$, that is,
$$
\sup_{h \in (0,1]} h^{-N} |f(h)| \lesssim 1,
$$
for all non-negative integers $N$, then we write $f(h) = \cO(h^\infty)$, $h \in (0,1]$.

The Fourier transform of $f \in \mathscr S(\rr d)$ (the Schwartz space) is normalized as
$$
\mathscr{F} f(\xi) = \wh f(\xi) = \int_{\rr d} f(x) e^{- i x \cdot \xi} \dd x, \quad \xi \in \rr d,
$$
where $x \cdot \xi$ denotes the inner product on $\rr d$.
The Japanese bracket is defined by $\la x \ra = \sqrt{1+|x|^2}$, $x \in \rr d$.
We use the family of seminorms on $\cS(\rr d)$ defined by
\begin{equation}\label{seminorm1}
\rho_{m,k}(g) = \sum_{|\alpha| \leq m} \sup_{y \in \rr d} \eabs{y}^k \left| \pd \alpha g(y) \right|, \quad m,k \in \no.
\end{equation}
The symbol $(\cdot,\cdot)$ is used for the $L^2$-inner product as well as the conjugate linear action of $\cS'(\rr d)$ on $\cS(\rr d)$.
The space $C_b^\infty(\rr d)$ is the space of smooth functions $f$ such that $\pd \alpha f \in L^\infty$ for all $\alpha \in \nn d$.
We use $D_{x_j}=-i \partial/\partial x_j$ and its multi-index extension.

In the following we recall some notions of time-frequency analysis and pseudodifferential operators, cf. e.g. \cite{Folland1,Grigis1,Grochenig1,Hormander0,Martinez1,Nicola1,Shubin1,Zworski1}. 
Let $u \in \cS'(\rr d)$ and $\varphi \in \cS(\rr d) \setminus \{ 0 \}$. The \emph{short-time Fourier (or Gabor) transform} (STFT)  $V_\varphi u $ of $u$ with respect to the window function $\varphi$ is defined as
$$
V_\varphi u :\ \rr {2d} \rightarrow \co, \quad (x,\xi) \mapsto V_\varphi u(x,\xi) = (u, M_\xi T_x \varphi ),
$$
where $T_x$ denotes the translation operator $T_x\varphi(y)=\varphi(y-x)$ and $M_\xi$ the modulation operator $M_\xi \varphi(y)=e^{i y \cdot \xi} \varphi(y)$.
The joint phase space translation operator is denoted $\Pi(z)=M_\xi T_x$, $z=(x,\xi) \in \rr {2d}$.
We have $V_\fy u \in C^\infty(\rr {2d})$ and there exists $N \in \no$ such that $$
|V_\fy u(z)| \lesssim \eabs{z}^N, \quad z \in \rr {2d}.
$$

Let $\psi \in \cS(\rr d)$ satisfy $\| \psi \|_{L^2}=1$.
The Moyal identity
\begin{equation*}
(u,g)= (2\pi)^{-d} \int_{\rr {2d}} V_\psi u(z) \, \overline{V_\psi g(z)} \, \dd z, \quad g \in \cS(\rr d), \quad u \in \cS'(\rr d),
\end{equation*}
is sometimes written
\begin{equation*}
u = (2\pi)^{-d} \int_{\rr {2d}} V_\psi u(x,\xi) \, M_\xi T_x \psi \, \dd x \, \dd \xi, \quad u \in \cS'(\rr d),
\end{equation*}
with action understood to take place under the integral. In this form it is a left inversion formula for the STFT.

Let $a \in C^\infty (\rr {2d})$. Then $a$ is a \emph{Shubin symbol} of order $m \in \ro$, denoted $a\in G^m$, if for all $\alpha,\beta \in \nn d$ there exists a constant $C_{\alpha\beta}>0$ such that
\begin{equation}\label{eq:shubinineq}
	|\partial_x^\alpha \partial_\xi^\beta a(x,\xi)|\leq C_{\alpha\beta}\langle (x,\xi)\rangle^{m-|\alpha|-|\beta|}, \quad x,\xi \in \rr d.
\end{equation}
The Shubin symbols form a Fr\'echet space where the seminorms are given by the smallest possible constants in \eqref{eq:shubinineq}.

For $a \in G^m$ and $t \in \ro$, a pseudodifferential operator in the $t$-quantization is defined by
\begin{equation*}
a^t(x,D_x) u(x)
= \int_{\rr {2d}} e^{i(x-y) \cdot \xi} a\big(tx+(1-t)y,\xi\big) \, u(y) \, \dd y \, \dbar \xi, \quad u \in \cS(\rr d),
\end{equation*}
when $m<-d$, where we use the convention $\dbar \xi=(2\pi)^{-d} \dd \xi$. The definition extends to general $m \in \ro$ by means of the following regularization procedure.
For $\psi \in C_c^\infty(\rr d)$ which equals one in a neighborhood of the origin one defines for $u \in \cS(\rr d)$
\begin{equation}\label{regularization0}
a^t(x,D_x) u(x)
= \lim_{\ep \rightarrow +0} \int_{\rr {2d}} \psi(\ep\xi) e^{i(x-y) \cdot \xi} a\big(tx+(1-t)y,\xi\big) \, u(y) \, \dd y \, \dbar \xi,
\end{equation}
which after integration by parts gives an absolutely convergent integral (cf. \cite{Nicola1,Shubin1}).

The operator $a^t(x,D_x)$ acts continuously on $\cS(\rr d)$ and extends by duality uniquely to a continuous operator on $\cS'(\rr d)$.
The $t$-quantization for general $t \in \ro$ contains as special cases the Kohn--Nirenberg quantization ($t=1$) denoted $a^1(x,D)= a(x,D)$,
and the Weyl quantization ($t=1/2$), denoted $a^{1/2}(x,D)=a^w(x,D)$, which we will use most of the time.
We will often dilate symbols with a positive parameter $h$, which is denoted $a_h(z)=a(hz)$, $z \in \rr {2d}$.

The Weyl product is the symbol product corresponding to composition of operators, $(a \wpr b)^w(x,D) = a^w(x,D) \, b^w (x,D)$.
The Weyl calculus enjoys the property $a^w(x,D)^* = \overline{a}^w(x,D)$, where $a^w(x,D)^*$ denotes the formal adjoint.
The Wigner distribution is defined by
\begin{align*}
W(f,g) (x,\xi) & = \int_{\rr d} f(x+y/2) \, \overline{g(x-y/2)} \, e^{-i y \cdot \xi} \, \dd y, \\
& \qquad \qquad x,\xi \in \rr d, \quad f,g \in \cS(\rr d),
\end{align*}
and we write $W(f)=W(f,f)$. It appears in the Weyl calculus in the formula
\begin{equation}\label{weylwigner}
(a^w(x,D) \, f, g) = (2 \pi)^{-d} (a,W(g,f) ), \quad a \in \cS'(\rr {2d}), \quad f,g \in \cS(\rr d).
\end{equation}

\section{Two notions of global microlocal singularities}

\label{sec:wfintro}

\subsection{The Gabor wave front set}

\begin{defn}
\label{def:wfgl}
Let $u \in \cS'(\rr d)$. The \emph{Gabor wave front set} $\WF(u)$ is defined by its complement: a point $(x,\xi)\in \rr {2d} \setminus \{0\}$ is not in $\WF(u)$ if one of the following equivalent conditions is met (cf. \cite[Proposition~6.8]{Hormander1} and \cite[Proposition~2.7, Corollary~3.3, Corollary~4.3]{Rodino1}):

\begin{itemize}
	\item There exists $a\in G^0$ such that $a^w(x,D) u \in \cS (\rr d)$ and $a$ is non-characteristic at $(x,\xi)$,
which means that $|a(sx,s\xi)|$ is lower bounded by a positive number for $s>0$ sufficiently large.
	\item For one (equivalently all) $\varphi \in \cS (\rr d)\setminus\{0\}$ there exists an open cone $\Gamma\subseteq \rr {2d} \setminus \{ 0 \}$ containing $(x,\xi)$ such that
$$
\sup_{z \in \Gamma} \, \eabs{z}^N |V_\varphi u(z)| \lesssim 1, \quad N \geq 0.
$$
	\item For one (equivalently all) choice of $\varphi \in \cS (\rr d)\setminus\{0\}$ and lattice $\Lambda=\alpha \zz d \times \beta \zz d \subseteq \rr {2d}$ where $\alpha,\beta>0$, such that $\{\Pi(\lambda) \fy \}_{\lambda \in \Lambda}$ constitutes a Gabor frame for $L^2(\rr d)$ (cf. \cite{Grochenig1}), there exists an open cone $\Gamma\subseteq \rr {2d} \setminus \{ 0 \}$ containing $(x,\xi)$ such that
$$
\sup_{\lambda \in \Gamma \cap \Lambda} \eabs{\lambda}^N |V_\varphi u(\lambda)| \lesssim 1, \quad N \geq 0.
$$
\end{itemize}
\end{defn}

Thus $\WF(u)$ is a closed conic set in $\rr {2d} \setminus \{ 0 \}$.
We list some of the main results for the Gabor wave front set (cf. \cite{Schulz1}). 

\begin{prop}
\cite[Proposition~2.4]{Hormander1}
Let $u \in \cS'(\rr d)$. Then
$$
\WF(u) = \emptyset \quad \Longleftrightarrow \quad u \in \cS (\rr d).
$$
\end{prop}

For each linear symplectic map $\chi$ on $T^*(\rr d)$, there exists a (unique up to a phase factor) unitary operator $U_\chi$ on $L^2(\rr d)$, such that
$$
U_\chi^*a^w(x,D)U_\chi=(a \circ \chi)^w(x,D).
$$
\begin{prop}\label{prop:sympinv}
\cite[Proposition~2.2]{Hormander1}
Let $u \in \cS' (\rr d)$ and let $\chi$ be a linear symplectic map. Then
$$
\WF(U_\chi u) = \chi \WF(u).
$$
\end{prop}

\begin{prop}
\cite[Proposition~2.5]{Hormander1}, \cite[Proposition~1.9]{Rodino1}
If $u \in \mathscr S'(\rr d)$ and $a \in G^m$ then
\begin{align*}
WF_G( a^w(x,D) u) \quad & \subseteq \quad WF_G(u) \bigcap \, \conesupp (a) \nonumber \\
& \subseteq \quad WF_G(u) \quad \subseteq \quad WF_G( a^w(x,D) u) \ \bigcup \ \charac (a),
\end{align*}
where $\conesupp (a)$ is the set of all
$z \in \rr {2d} \setminus \{ 0 \}$ such that any conic open set $\Gamma_z \subseteq \rr {2d} \setminus \{ 0 \}$ containing $z$ satisfies:
$$
\overline{\supp (a) \cap \Gamma_z} \quad \mbox{is not compact in} \quad \rr {2d}.
$$
\end{prop}
\begin{thm}
\cite[Proposition~4.1]{Rodino1}
If $u \in \cS'(\rr d)$ and $a \in S_{0,0}^0$, i.e. $a \in C^\infty(\rr {2d})$ and $\pd \alpha a \in L^\infty$ for all $\alpha \in \nn {2d}$, then
$$
WF_G( a^w(x,D) u) \subseteq WF_G(u) \bigcap \, \conesupp(a).
$$
\end{thm}

\subsection{The homogeneous wave front set}

Next we introduce the homogeneous wave front set of Nakamura \cite{Nakamura1}.
Here the semiclassical dilation parameter $h$ is introduced in both phase space variables simultaneously.

\begin{defn}\label{hwfdef}
Let $u\in\cS' (\rr d)$. A point $(x,\xi)\in\rr {2d}\setminus\{0\}$ is not in the homogeneous wave front set $\HWF(u)$ if
there exists $a\in C_c^\infty (\rr {2d})$ with $a(x,\xi)=1$ such that
\begin{equation}\label{eq:homwfdef}
\|a_h^w(x,D)u\|_{L^2} = \|a^w(hx,hD)u\|_{L^2}=\cO (h^\infty)\text{ for }h\in(0,1].
\end{equation}
\end{defn}

The homogeneous wave front set has been studied in \cite{Ito1,Nakamura1} where it is used in results on propagation of singularities for Schr\"odinger type equations with variable coefficients.

In order to be able to assume additonal properties of the test function symbol $a$,
we need a result saying that Definition \ref{hwfdef} does not depend on the test function symbol $a$.
More precisely: If \eqref{eq:homwfdef} is valid for some test function $a$ then it holds for any
test function supported in a small neighborhood of $(x,\xi)$.
To prove this we need to develop a pseudodifferential calculus for semiclassically dilated symbols in the phase space variables, which we call a global semiclassical pseudodifferential calculus.
Such a calculus does not seem to exist in the literature, so we devote Appendix \ref{appendix} to this topic.
There we prove in particular the aforementioned invariance result Theorem \ref{HWFinvariance}.

\subsection{The $h$-dependent STFT}

Here we introduce the STFT with respect to a small positive parameter $h$.
Note that the parameter appears in both variables and covariables, as opposed to the standard concept in semiclassical analysis (cf. \cite{Martinez1,Zworski1}) where only the covariables are dilated.

\begin{defn}\label{hSTFTdef}
Let $u \in \cS'(\rr d)$ and $\varphi(y)=\pi^{-d/4}e^{-|y|^2/2}$ for $y \in \rr d$.
For $h\in(0,1]$ and $x,\xi \in \rr d$, the $h$-dependent STFT (short: $h$STFT)
is defined by
$$
T_h u(x,\xi) = (2\pi)^{-d/2} h^{-d}  (u, T_{x/h} M_{\xi/h} \fy) \in C^\infty(\rr {2d}).
$$
\end{defn}

We list some properties of the $h$STFT, following \cite[Proposition 6.1]{Hormander1}, \cite[Theorems 11.2.3 and 11.2.5]{Grochenig1} and the corresponding semiclassical results in \cite{Martinez1}.

\begin{lem}\label{lem:hstftprop}
The $h$-dependent STFT has the following properties.
\begin{enumerate}
\item We have
\begin{equation}
\label{eq:hstftgauss}
T_hu(x,\xi)=(2\pi)^{-d/2}h^{-d} e^{i x \cdot \xi/h^2} V_\varphi u (x/h,\xi/h).
\end{equation}
For fixed $h \in (0,1]$, $T_h$ maps:
\begin{enumerate}
	\item $\cS(\rr d)$ continuously into $\cS(\rr {2d})$, and if $u \in \cS$ then $|T_h u(x,\xi)| \leq C_N h^{-d} \langle (x/h,\xi/h)\rangle^{-N}$ for every $N \in \no$;
	\item $\cS'(\rr d)$ continuously into
$(C^\infty \cap \cS')(\rr {2d})$, and if $u \in \cS'$ then $|T_h u(x,\xi)| \leq C_N h^{-d} \langle (x/h,\xi/h)\rangle^N$ for some $N \in \no$;
	\item $L^2(\rr d)$ isometrically into $L^2(\rr {2d})$.
\end{enumerate}

\item We can write
$$
T_h u (x,\xi)=(2\pi)^{-d/2}h^{-d} e^{-|\xi|^2/(2h^2)}(u*\varphi)(x/h-i\xi/h)
$$
and deduce that for $u \in \cS'(\rr d)$, $e^{|\xi|^2/(2h^2)}T_h u (x,\xi)$ is an entire function of $x-i\xi$.

\item We have the Moyal identity
\begin{equation}\label{Moyal1}
\begin{aligned}
(u,g) & = (2\pi)^{-d/2} h^{-d} \int_{\rr {2d}} T_h u(x,\xi) \left( T_{x/h} M_{\xi/h} \fy,g \right) \, \dd x \, \dd \xi, \\
& \qquad \qquad g \in \cS(\rr d), \quad u \in \cS'(\rr d).
\end{aligned}
\end{equation}

\end{enumerate}
\end{lem}

The following proposition gives a characterization of the Gabor wave front set in terms of the $h$STFT.
It will be an essential tool in the proof of the identity of the Gabor and the homogeneous wave front sets.

\begin{prop}\label{WFchar}
Let $u \in \cS'(\rr d)$. Then $0 \neq z_0 \in \rr {2d} \setminus \WF(u)$ if and only if there exists an open set $U \ni z_0$ such that
$$
\|T_h u |_U \|_{L^\infty(\rr {2d})} = \mathcal{O}(h^\infty), \quad h \in (0,1].
$$
\end{prop}

This result can be visualized in phase space, see Figure \ref{fig:WFchar}.

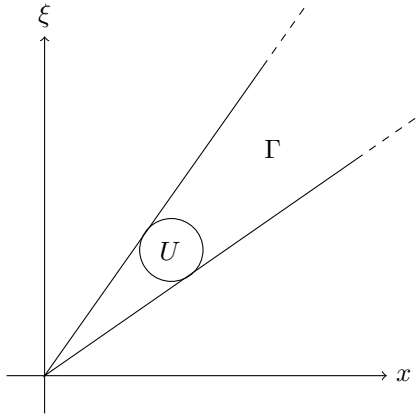
\begin{figure}[!h]
\begin{center}
\begin{tikzpicture}
    \draw[->] (-0.5,0)--(4.5,0) node[right]{$x$};
    \draw[->] (0,-0.5)--(0,4.5) node[above]{$\xi$};

\draw[-] ({5 * cos(55.0)},{5 * sin(55.0)}) -- ({0.0 * cos(55.0)},{0.0 * sin(55.0)}) arc (55:35:0.0) -- ({5 * cos(35.0)},{5 * sin(35.0)});
\draw[-,dashed]({5 * cos(35.0)},{5 * sin(35.0)}) -- ({6 * cos(35.0)},{6 * sin(35.0)});
\draw[-,dashed]({5 * cos(55.0)},{5 * sin(55.0)}) -- ({6 * cos(55.0)},{6 * sin(55.0)});

	\foreach \h in {0.6}{
	\draw ({1/\h},{1/\h}) circle ({0.25/\h});
}
\node at (1.65,1.65) {$U$};
\node at (3.0,3.0) {$\Gamma$};

\end{tikzpicture}
\end{center}
\caption{A comparison of the description of $\WF$ via $T_h$ and $V_\varphi$.}
\label{fig:WFchar}
\end{figure}

\begin{proof}
Suppose $0 \neq z_0 \notin \WF(u)$, i.e.
\begin{equation}
\label{eq:wfcharineq}
\sup_{z \in \Gamma_{z_0}} \eabs{z}^N |V_\varphi u(z)| < \infty \quad \forall N \geq 0,
\end{equation}
for an open conic set $\Gamma_{z_0} \subseteq \rr {2d} \setminus \{ 0 \}$ containing $z_0$.
Let $\Gamma' \subseteq \Gamma_{z_0}$ be an open cone containing $z_0$
such that $\overline{\Gamma' \cap S_{2d-1}} \subseteq \Gamma_{z_0}$,
and let $\ep>0$ be so small that
\begin{equation*}
U = \left\{ \frac{z |z_0|}{|z|} + w \in \rr {2d} \setminus \{ 0 \}: \, z \in \Gamma', \, |w| < \ep \right\} \subseteq \Gamma_{z_0}
\end{equation*}
and $\inf_{z \in U} |z| > 0$.
Then $U$ is open and contains $z_0$.
If $z \in U$ then $z/h \in \Gamma_{z_0}$ for all $h \in (0,1]$, so from \eqref{eq:hstftgauss} 
and \eqref{eq:wfcharineq} we obtain for any $N \geq 0$
\begin{align*}
\sup_{z \in U} |T_h u (z)|
& \lesssim h^{-d} \sup_{z \in U} \eabs{z/h}^{-N}
\lesssim h^{N-d}, \quad h \in (0,1].
\end{align*}

Suppose on the other hand $0 \neq z_0 \in U$ where $U$ is open and
\begin{equation*}
\sup_{z \in U} |T_h u (z)| \lesssim h^{N}, \quad h\in(0,1], \quad N \geq 0.
\end{equation*}
Define
\begin{equation*}
\Gamma = \left\{ z \in \rr {2d} \setminus \{ 0 \}: \,  \frac{z |z_0|}{|z|} \in U \right\}
\end{equation*}
which is a conic open subset of $\rr {2d} \setminus \{ 0 \}$ containing $z_0$.
If $z \in \Gamma$, $|z| \geq |z_0|$ and $h=|z_0|/|z|$ then $z h \in U$, and we have for any $N \geq 0$
\begin{equation*}
|V_\varphi u(z)| = (2 \pi)^{d/2} h^d |T_h u(h z)|
\lesssim h^{N+d} \lesssim |z|^{-N}.
\end{equation*}
As $V_\varphi u$ is smooth and therefore bounded for finite arguments we have
\begin{equation*}
|V_\varphi u(z)| \lesssim \eabs{z}^{-N}, \quad z \in \Gamma, \quad N \geq 0,
\end{equation*}
which means that $z_0 \notin \WF(u)$.
\end{proof}

\section{Preparation for the main result}
\label{sec:prep}

In order to prove the equality of the global and the homogeneous wave front sets, we will need to estimate the action of pseudodifferential operators with dilated symbols of the form $a_h^t(y,D_y)$ on $T_{x/h} M_{\xi/h} \varphi$ where $\varphi$ is a Gaussian, with assumptions on the support of the symbol $a\in G^0$.
For this purpose we show a series of lemmas.

First we need a lemma concerning the formal transpose of the operator $^t K$, defined for $y,\eta \in \rr d$ by
\begin{align*}
^t K = \frac{-\left(y - i \eta \right) \cdot \nabla_y}{|y|^2 + |\eta|^2}.
\end{align*}
The operator $^t K$ is continuous on $C^\infty(\rr {2d} \setminus \overline{B_\ep})$ for any $\ep>0$.
Its formal transpose is defined by
\begin{align*}
K f (y,\eta) = \nabla_y \cdot \left( \frac{\left(y - i \eta \right) f(y,\eta)}{|y|^2 + |\eta|^2} \right)
\end{align*}
and acts continuously on $C^\infty(\rr {2d} \setminus \overline{B_\ep})$ for any $\ep>0$.

\begin{lem}\label{Kestimate}
Let $\ep>0$ and $f \in C_b^\infty(\rr {2d})$.
For $N \in \no$ and $\beta \in \nn d$ arbitrary, the differentiated iterated formal transpose of $^t K$ obeys the estimate, for some $C_{N,\beta} > 0$,
\begin{equation}\label{KNestimate}
|\pdd \eta \beta K^N f(y,\eta) | \leq C_{N,\beta} \eabs{(y,\eta)}^{-N}, \quad y,\eta \in \rr d, \quad |y|^2+|\eta|^2 \geq \ep^2.
\end{equation}
\end{lem}

\begin{proof}
Set $z = y - i \eta \in \cc d$.
For $f \in C_b^\infty(\rr {2d})$ we have by induction if $z \neq 0$
\begin{align}\label{KNformula}
K^N f(z) = \sum_{j_1, \dots, j_N=1}^d \sum_{\substack{\alpha_0,\dots,\alpha_N \in \nn d:\\|\alpha_0|+\cdots+|\alpha_N|=N}} k_{\alpha} \, \partial_y^{\alpha_1} \left( \frac{z_{j_1}}{|z|^2} \right) \cdots \partial_y^{\alpha_N} \left( \frac{z_{j_N}}{|z|^2} \right) \partial_y^{\alpha_0} f(z)
\end{align}
where $k_\alpha=k_{\alpha_0,\dots,\alpha_N}$ is either zero or one.
Again by induction we have if $z \neq 0$
\begin{equation*}
\partial_\eta^{\beta} \partial_y^{\alpha} \left( \frac{z_j}{|z|^2} \right) = \frac{p_{\alpha,\beta,j}(y,\eta)}{|z|^{4 (|\alpha|+|\beta|)}}, \quad \alpha,\beta \in \nn d, \quad 1 \leq j \leq d,
\end{equation*}
where $p_{\alpha,\beta,j}$ is a polynomial such that $\deg(p_{\alpha,\beta,j}) \leq 3 (|\alpha|+|\beta|)-1$.
Using the assumption $|z|^2 \geq \ep^2$ we may thus estimate
\begin{equation*}
\left| \partial_\eta^{\beta} \partial_y^{\alpha} \left( \frac{z_j}{|z|^2} \right) \right|
 \leq C_{\alpha,\beta,j} \eabs{z}^{-|\alpha|-|\beta|-1}, \quad \alpha,\beta \in \nn d, \quad 1 \leq j \leq d.
\end{equation*}
Inserted into $\pdd \eta \beta$ acting on
\eqref{KNformula} this proves \eqref{KNestimate}.
\end{proof}

The following lemma estimates the action of a pseudodifferential operator in the $0$-quantization whose symbol vanishes around a given point. The operator acts on $T_{x/h} M_{\xi/h} \varphi$ where $\varphi$ is a Gaussian and $(x,\xi)$ belongs to a small neighborhood of the given point.

\begin{mainlem}\label{lem:homwf}
Let $\varphi(y)=e^{-|y|^2/2}$ for $y \in \rr d$, and suppose $a \in G^0$ vanishes in a neighborhood of $(x_0,\xi_0) \in \rr {2d}$.
For $\delta>0$ sufficiently small we have for any $m,k \in \no$
\begin{equation}\label{seminormestimate1}
\begin{aligned}
\rho_{m,k}\Big( a^0(hy,hD_y) \left( T_{x/h} M_{\xi/h} \varphi \right)\Big)
& \leq C_{N,m,k} h^N, \quad 0 < h \leq 1, \quad N \geq 0, \\
& \quad (x,\xi) \in B_\delta(x_0,\xi_0).
\end{aligned}
\end{equation}
\end{mainlem}

\begin{proof}
Let $\psi \in C_c^\infty(\rr d)$ equal one in a neighborhood of the origin.
By \eqref{regularization0}
\begin{align*}
& a^0(hy,hD_y) \left( T_{x/h} M_{\xi/h} \varphi \right)(y) \\
& = \lim_{\ep \rightarrow 0+} \int_{\rr {2d}} \psi(\ep \eta) \, e^{i(y-z) \cdot \eta} a(hz,h\eta) \, T_{x/h} M_{\xi/h} \varphi (z) \, \dd z \, \dbar \eta \\
&=h^{-2d} \lim_{\ep \rightarrow 0+} \int \psi(\ep \eta/h)
\, e^{i\Phi(x,z,\eta,\xi)/h^2} \, e^{iy \cdot \eta/h} a(z,\eta) \, \dd z \, \dbar \eta
\end{align*}
where $\Phi(x,z,\eta,\xi) = -z \cdot \eta+(z-x)\cdot \xi+i|z-x|^2/2$.

According to the assumptions there exists $\delta>0$ such that $\supp(a) \cap B_{2 \delta}(x_0,\xi_0) = \emptyset$.
Let $(x,\xi) \in B_\delta(x_0,\xi_0)$.
To show the seminorm estimate \eqref{seminormestimate1} for any $m,k \in \no$
we let $\alpha, \beta \in \nn d$ be arbitrary,
and estimate the supremum over $y \in \rr d$ of
the modulus of
\begin{equation}\label{seminorm2}
I = y^\beta \pdd y \alpha a^0(hy,hD_y) \, (T_{x/h} M_{\xi/h} \varphi) (y).
\end{equation}
To estimate $|I|$ we use Lemma \ref{Kestimate}. If we define
\begin{align*}
^tK & = \frac{-(z-x - i(\eta-\xi)) \cdot \nabla_z}{|z-x|^2+|\eta-\xi|^2}
\end{align*}
then $^t K$ is well defined for $(z,\eta) \notin B_{2\delta}(x_0,\xi_0)$, since
$$
|(z,\eta) - (x,\xi)| = | (z,\eta) - (x_0,\xi_0) - \left( (x,\xi) - (x_0,\xi_0) \right) | \geq 2 \delta - \delta = \delta.
$$
Moreover, $^t Ke^{i \Phi/h^2} = h^{-2}e^{i \Phi/h^2}$.
Integration by parts gives for any $N \in \no$
\begin{align}
I
& = h^{2(N-d)} \lim_{\ep \rightarrow 0+} \int_{(z,\eta) \in \rr {2d} \setminus B_{2\delta}(x_0,\xi_0)}
\psi(\ep \eta/h) \, e^{i\Phi(x,z,\eta,\xi)/h^2} \nonumber \\
& \qquad \qquad \qquad \qquad \qquad \qquad \qquad \times y^\beta (i \eta/h)^\alpha e^{iy \cdot \eta/h} K^N \left( a(z,\eta) \right) \dd z \, \dbar \eta \nonumber \\
& = h^{2(N-d)} \lim_{\ep \rightarrow 0+} \int_{(z,\eta) \in \rr {2d} \setminus B_{2\delta}(x_0,\xi_0)}
\psi(\ep \eta/h) \, e^{i\Phi(x,z,\eta,\xi)/h^2} \nonumber \\
& \qquad \qquad \qquad \qquad \qquad \times (i \eta/h)^\alpha h^{|\beta|} D_\eta^\beta ( e^{iy \cdot \eta/h} )
K^N \left( a(z,\eta) \right) \, \dd z \, \dbar \eta \nonumber \\
& = h^{2(N-d)+|\beta|-|\alpha|} \, i^{|\alpha|}
\lim_{\ep \rightarrow 0+} \int_{(z,\eta) \in \rr {2d} \setminus B_{2\delta}(x_0,\xi_0)} e^{iy \cdot \eta/h} \nonumber \\
& \qquad \times (-D_\eta)^\beta \left( \psi(\ep \eta/h) \, e^{i\Phi(x,z,\eta,\xi)/h^2} \, \eta^\alpha \,
K^N \left( a(z,\eta) \right) \right) \dd z \, \dbar \eta. \label{I1}
\end{align}

By means of Lemma \ref{Kestimate} we may estimate the integrand as
\begin{align*}
& \left| (-D_\eta)^\beta \left( \psi(\ep \eta/h) \, e^{i\Phi(x,z,\eta,\xi)/h^2} \, \eta^{\alpha} \,
K^N \left( a(z,\eta) \right) \right) \right| \\
& \leq \sum_{|\beta_1|+|\beta_2|+|\beta_3|+|\beta_4|=|\beta|}
C_{\beta_1,\beta_2,\beta_3,\beta_4}
\Big| \partial_\eta^{\beta_1} \left( \psi(\ep \eta/h) \right)
\partial_\eta^{\beta_2} \left( e^{i\Phi(x,z,\eta,\xi)/h^2} \right) \\
& \qquad \qquad \times \partial_\eta^{\beta_3} \left( \eta^{\alpha}  \right)
\partial_\eta^{\beta_4} K^N \left( a(z,\eta) \right) \Big| \\
& \leq C_{N,\alpha,\beta} \sum_{|\beta_1|+|\beta_2|+|\beta_3|+|\beta_4|=|\beta|}
h^{-|\beta_1|-2|\beta_2|} \ep^{|\beta_1|}
\left| \partial^{\beta_1} \psi(\ep \eta/h) \right| \eabs{z}^{|\beta|} e^{-|z-x|^2/(2 h^2)} \\
& \qquad \qquad \times \eabs{\eta}^{|\alpha|}
\eabs{(z-x,\eta-\xi)}^{-N} \\
& \leq C_{N,\alpha,\beta} \eabs{x}^{|\beta|} \eabs{\xi}^{|\alpha|} h^{-2 |\beta|} \\
& \qquad \times \sum_{|\beta_1|+|\beta_2|+|\beta_3|+|\beta_4|=|\beta|}
\ep^{|\beta_1|} \left| \partial^{\beta_1} \psi(\ep \eta/h) \right|
\eabs{z-x}^{|\beta|+N} e^{-|z-x|^2/2}
\eabs{\eta-\xi}^{|\alpha|-N}.
\end{align*}
Provided $N > d + |\alpha|$ the integral \eqref{I1} converges and we have
\begin{equation*}
|I| \leq C_{N,\alpha,\beta} h^{2(N-d) - |\alpha| - |\beta|}
\end{equation*}
for some constant $C_{N,\alpha,\beta}$ that does not depend on $h$.
Combined with \eqref{seminorm2} this shows that
\begin{equation*}
\sup_{y \in \rr d} \left| y^\beta \pdd y \alpha a^0(hy,hD_y) \left( T_{x/h} M_{\xi/h} \varphi \right)(y) \right|
\leq C_{N,\alpha,\beta} h^{2(N-d)-|\alpha|-|\beta|}
\end{equation*}
holds for $h \in (0,1]$, any $N \geq 0$ and any $\alpha ,\beta \in \nn d$.
\end{proof}

The next lemma is similar in spirit to the preceding one with opposite assumptions.
It assumes that the symbol is compactly supported
in a small ball, and $(x,\xi)$ stays outside a larger ball.
In contrast to the $0$-quantization of Main Lemma \ref{lem:homwf}, we use the Weyl quantization here.

\begin{mainlem}\label{L2estimate}
Let $\varphi(y)=e^{-|y|^2/2}$ for $y \in \rr d$, $(x_0,\xi_0) \in \rr {2d}$ and $\ep>0$.
Suppose $a \in C_c^\infty(\rr{2d})$ with
\begin{equation}\label{smallsupport1}
\supp(a) \subseteq B_{\ep/4}(x_0,\xi_0).
\end{equation}
For all $(x,\xi)\in\rr{2d}$ such that
\begin{equation}\label{ballcomplement1}
|x-x_0|^2 + |\xi-\xi_0|^2 \geq \ep^2 > 0,
\end{equation}
there exists for all $N\in\no$ a constant $C_N>0$ such that
\begin{equation*}
\| a^w(hy,hD_y) \left( T_{x/h} M_{\xi/h} \varphi \right) \|_{L^2}^2 \leq C_N h^{2N} \eabs{(x,\xi)}^{-N}, \quad h \in (0,1].
\end{equation*}
\end{mainlem}

\begin{proof}
We write $a_h(z)=a(h z)$, $z \in \rr {2d}$, and obtain by means of \eqref{weylwigner}
\begin{align*}
\| a_h^w(y,D_y) \left( T_{x/h} M_{\xi/h} \varphi \right) \|_{L^2}^2
& = ( a_h^w(y,D_y)^* a_h^w(y,D_y) \left( T_{x/h} M_{\xi/h} \varphi \right), T_{x/h} M_{\xi/h} \varphi)_{L^2} \\
& = (2 \pi)^{-d} \left( \overline{a}_h \wpr a_h, W \left( T_{x/h} M_{\xi/h} \varphi \right) \right)_{L^2(\rr {2d})}.
\end{align*}
We have by \cite[Chapter 18.5]{Hormander0}
\begin{equation*}
\overline{a}_h \wpr \, a_h (y,\eta) = \pi^{-2d} \int \overline{a_h (z,\zeta)} \, a_h(t,\theta) \,
e^{2 i \left( (\theta-\eta) \cdot (z-y) - (t-y) \cdot (\zeta-\eta) \right) }
\, \dd z \, \dd \zeta \, \dd t \, \dd \theta.
\end{equation*}
Denote by $\mathscr F_2$ the partial Fourier transformation of a function on $\rr d \oplus \rr d$ with respect to the second $\rr d$ variable.
Since $W(f)= \mathscr F_2(f \otimes \overline{f} \circ \kappa)$ for $f \in \cS(\rr d)$
and $\kappa(x,y)=(x+y/2,x-y/2)$,
we have
\begin{align*}
& \int_{\rr {2d}} W \left( T_{x/h} M_{\xi/h} \varphi \right)(y,\eta) \,
e^{2 i (\eta \cdot (t-z) - y \cdot (\theta-\zeta))} \, \dd \eta \, \dd y \\
& = (2 \pi)^d \, \int_{\rr d} (T_{x/h} M_{\xi/h} \varphi \otimes \overline{T_{x/h} M_{\xi/h} \varphi} \circ \kappa) (y,2(t-z)) \, e^{- 2 i y \cdot (\theta-\zeta)} \, \dd y \\
& = 2^d \pi^{3d/2} \exp \left(-|\theta-\zeta|^2 -|t-z|^2 + 2 i \left( \xi \cdot(t-z) - x \cdot (\theta-\zeta) \right)/h \right).
\end{align*}
Using the fact that $W(T_{x/h} M_{\xi/h})$ is real-valued we obtain
\begin{align*}
& \| a_h^w(y,D_y) \left( T_{x/h} M_{\xi/h} \varphi \right) \|_{L^2}^2 \\
& = \pi^{-3d/2} \int \overline{a_h (z,\zeta)} \, a_h(t,\theta) \,
e^{2 i h^{-1}\left( \xi \cdot(t-z) - x \cdot (\theta-\zeta) \right)  + 2 i (\theta \cdot z - \zeta \cdot t) - |\theta-\zeta|^2 -|t-z|^2} \, \dd z \,\dd \zeta \, \dd t \, \dd \theta \\
& = \pi^{-3d/2} h^{-4d} \int \overline{a(z,\zeta)} \, a(t,\theta) \,
e^{ 2 i h^{-2} \left( \xi \cdot(t-z) - x \cdot (\theta-\zeta) + \theta \cdot z - \zeta \cdot t  + i |\theta-\zeta|^2/2 + i |t-z|^2/2 \right) }
\,\dd z \, \dd \zeta \, \dd t \, \dd \theta.
\end{align*}

Writing $a(z,\zeta)=a_0(z-x_0,\zeta-\xi_0)$ and changing variables yield
\begin{multline}\label{stationaryphase1}
\pi^{3d/2} \| a_h^w(y,D_y) \left( T_{x/h} M_{\xi/h} \varphi \right) \|_{L^2}^2\\
 = h^{-4d} \int \overline{a_0(z,\zeta)} \, a_0(t,\theta) \,
e^{i\omega f (z,\zeta; t, \theta)}
\, \dd z \, \dd \zeta \, \dd t \, \dd \theta
\end{multline}
where we have written
\begin{align*}
\omega&= \lambda h^{-2}, \quad \lambda=\eabs{(x,\xi)-(x_0,\xi_0)},\\
f (z,\zeta; t, \theta) & = 2 \lambda^{-1} \Big( (x_0-x) \cdot (\theta-\zeta) - (\xi_0-\xi) \cdot(t-z) \\
& \qquad \qquad + \theta \cdot z - \zeta \cdot t  + i |\theta-\zeta|^2/2 + i |t-z|^2/2 \Big).
\end{align*}
Then
\begin{align*}
|f'|^2 & = |\nabla_z f|^2 + |\nabla_\zeta f|^2 + |\nabla_t f|^2 + |\nabla_\theta f|^2 \\
& = 4 \lambda^{-2} \Big( |\xi_0-\xi+\theta -i(t-z)|^2 + |\xi-\xi_0-\zeta +i(t-z)|^2  \\
& \qquad \qquad + |x-x_0-t - i(\theta-\zeta) |^2 + |x_0-x+z + i(\theta-\zeta) |^2 \Big) \\
& = 4 \lambda^{-2} \Big( 2 |\xi_0-\xi|^2 + 2 |x_0-x|^2 + |z|^2+|\zeta|^2+|t|^2+|\theta|^2 \\
& \qquad \qquad + 2 ((\theta+\zeta) \cdot (\xi_0-\xi) + (z+t) \cdot (x_0-x) )
+ 2 |\theta-\zeta |^2 + 2|t-z|^2 \Big).
\end{align*}

When $(z,\zeta), (t, \theta) \in \supp(a_0)$
we have by the assumptions \eqref{smallsupport1} and \eqref{ballcomplement1}
\begin{align*}
|(\theta+\zeta) \cdot (\xi_0-\xi) + (z+t) \cdot (x_0-x)| \leq  \frac1{2} \left( |x-x_0|^2 + |\xi-\xi_0|^2 \right),
\end{align*}
and thus when $(z,\zeta), (t, \theta) \in \supp(a_0)$
\begin{align*}
|f'|^2 + \im f
& \geq 4 \lambda^{-2} \left( |\xi_0-\xi|^2 + |x_0-x|^2 \right) \\
& \geq 2 \min(1,\ep)^2.
\end{align*}
Combining this with \eqref{stationaryphase1} and H\"ormander's nonstationary phase result \cite[Theorem 7.7.1]{Hormander0}, we get the following conclusion:
For any $N \in \no$ there exists $C_N>0$ that does not depend on $(x,\xi) \in \rr {2d}$, provided \eqref{ballcomplement1} is satisfied, such that
\begin{align*}
& \| a_h^w(y,D_y) \left( T_{x/h} M_{\xi/h} \varphi \right) \|_{L^2}^2 \\
& \leq C_N h^{-4d} \omega^{-N}
\sum_{|\alpha| \leq N} \sup_{\supp(a_0) \times \supp(a_0)} |\pd \alpha (\overline{a_0} \otimes a_0) | \left( |f'|^2 + \im f \right)^{|\alpha|/2-N} \\
& \lesssim C_N h^{2(N-2d)} \lambda^{-N} \\
& \lesssim C_N h^{2(N-2d)} \eabs{(x_0,\xi_0)}^{N} \eabs{(x,\xi)}^{-N} \\
& \leq C_N h^{2(N-2d)} \eabs{(x_0,\xi_0)}^{N} \eabs{(x,\xi)}^{-(N-2d)}.
\end{align*}
\end{proof}

\section{Proof of the main result}

\label{sec:main}
\begin{thm}\label{WFequality}
If $u \in \cS'(\rr d)$ then $\WF(u) = \HWF(u)$.
\end{thm}
\begin{proof}
Suppose $0 \neq (x_0,\xi_0) \notin \HWF(u)$.
By Definition \ref{hwfdef} and Theorem \ref{HWFinvariance} (see Appendix \ref{appendix}) there exists $a \in C_c^\infty(\rr {2d})$ with $a=1$ in a neighborhood of $(x_0,\xi_0)$ such that
\begin{equation*}
\| a^w (hy,hD_y) u \|_{L^2(\rr d)} \lesssim h^N, \quad h \in (0,1], \quad N \geq 0.
\end{equation*}
By Proposition \ref{kohnnirenberg} (see Appendix \ref{appendix}) we may assume that
\begin{equation*}
\| a(hy,hD_y) u \|_{L^2(\rr d)} \lesssim h^N, \quad h \in (0,1], \quad N \geq 0.
\end{equation*}
Write $u=a(hy,hD_y)u+(1-a)(hy,hD_y)u$.
Using Definition \ref{hSTFTdef} we have with $\varphi(y)=\pi^{-d/4} e^{- |y|^2/2}$ for any $(x,\xi) \in \rr {2d}$
\begin{equation}\label{hSTFTestimate1}
\begin{aligned}
| T_h (a(hy,hD_y) u) (x,\xi) |
& \leq (2 \pi)^{-d/2} h^{-d} \| a(hy,hD_y) u \|_{L^2(\rr d)} \, \| \fy \|_{L^2} \\
& \lesssim h^{N-d}, \quad h \in (0,1], \quad N \geq 0.
\end{aligned}
\end{equation}

From $(a(x,D)u,g) = (u,\overline{a}^0(x,D) g)$ for $g \in \cS$, and Main Lemma \ref{lem:homwf},
we obtain for some $m,k \in \no$
\begin{align*}
|(T_h(1-a)&(hy,hD_y)u)(x,\xi)|\\
& = (2 \pi)^{-d/2} h^{-d} |( (1-a)(hy,hD_y)u, T_{x/h} M_{\xi/h} \varphi )| \\
& = (2 \pi)^{-d/2} h^{-d} |( u, (1-\overline{a})^0(hy,hD_y) \, (T_{x/h} M_{\xi/h} \varphi) )| \\
& \lesssim h^{-d} \rho_{m,k}\big( (1-\overline{a})^0(hy,hD_y) \, (T_{x/h} M_{\xi/h} \varphi)\big) \\
& \lesssim h^{N-d}, \quad h \in (0,1], \quad N \geq 0,
\end{align*}
for $(x,\xi) \in B_\delta(x_0,\xi_0)$ with $\delta>0$ sufficiently small.
Combining with \eqref{hSTFTestimate1} we may conclude
\begin{equation*}
\sup_{(x,\xi) \in B_\delta(x_0,\xi_0)} |T_h u(x,\xi)| \lesssim h^N, \quad h \in (0,1], \quad N \geq 0,
\end{equation*}
and by Proposition \ref{WFchar} it follows that $(x_0,\xi_0) \notin \WF(u)$.
This proves $\WF(u) \subseteq \HWF(u)$.

Suppose on the other hand that $0 \neq (x_0,\xi_0) \notin \WF(u)$.
By means of the Moyal identity \eqref{Moyal1} we obtain for $a \in C_c^\infty(\rr {2d})$
\begin{equation}\label{homogeneous1}
\begin{aligned}
 \| a^w &(hy,hD_y)u \|_{L^2} \\
&= \sup_{g \in \cS, \, \| g \|=1} \left| (a^w (hy,hD_y)u,g) \right| \\
& = \sup_{g \in \cS, \, \| g \|=1} \left| (u, \overline{a}^w (hy,hD_y)g) \right| \\
& \lesssim \sup_{g \in \cS, \, \| g \|=1}
h^{-d} \int_{\rr {2d}} \left| T_h u(x,\xi) \, ( a^w_h (y,D_y)(T_{x/h} M_{\xi/h} \varphi),g)_{L^2} \right| \, \dd x \, \dd \xi \\
& \lesssim h^{-d} \int_{\rr {2d}} \left| T_h u(x,\xi) \right| \, \| a_h^w (y,D_y)(T_{x/h} M_{\xi/h} \varphi) \|_{L^2} \, \dd x \, \dd \xi.
\end{aligned}
\end{equation}
By Proposition \ref{WFchar} there exists a neighborhood $U \ni (x_0,\xi_0)$ (that may be assumed to be relatively compact) such that
$$
\sup_{(x,\xi) \in U} \left| T_h u(x,\xi) \right| \lesssim h^N, \quad h \in (0,1], \quad N \geq 0.
$$
By the Calder\'on--Vaillancourt theorem (cf. e.g. \cite{Folland1}) we have for some $n \geq 0$ the estimate
\begin{align*}
\| a_h^w (y,D_y)(T_{x/h} M_{\xi/h} \varphi) \|_{L^2} \lesssim \| \varphi \|_{L^2} \sum_{|\alpha| \leq n} h^{|\alpha|} \sup \left| \pd \alpha a \right|.
\end{align*}
This gives
\begin{equation}\label{penultimastima}
\begin{aligned}
& \int_{U} \left| T_h u(x,\xi) \right| \, \| a_h^w (y,D_y)(T_{x/h} M_{\xi/h} \varphi) \|_{L^2} \, \dd x \, \dd \xi
\lesssim h^N, \\
& \qquad \qquad \qquad \qquad \qquad \qquad \qquad \qquad \qquad \quad h \in (0,1], \quad N \geq 0.
\end{aligned}
\end{equation}

It remains to estimate the integral over $\rr {2d}\setminus U$:
by Lemma \ref{lem:hstftprop} (1) (b) we have
$$
|T_h u(x,\xi)| \lesssim h^{-d} \eabs{(x/h,\xi/h)}^M \leq h^{-d-M} \eabs{(x,\xi)}^M, \quad h \in (0,1],
$$
for some $M \geq 0$.
Using Main Lemma \ref{L2estimate} this gives, provided $a \in C_c^\infty(\rr {2d})$ is supported in a sufficiently small
neighborhood of $(x_0,\xi_0)$,
\begin{align*}
& \int_{\rr {2d} \setminus U} \left| T_h u(x,\xi) \right| \, \| a_h^w (y,D_y)(T_{x/h} M_{\xi/h} \varphi) \|_{L^2} \, \dd x \, \dd \xi \\
& \lesssim h^{-d-M} \iint_{\rr {2d} \setminus U} \eabs{(x,\xi)}^M \, C_N h^N \eabs{(x,\xi)}^{-N/2} \, \dd x \, \dd \xi \\
& \lesssim h^{N-d-M},\quad h \in (0,1],
\end{align*}
provided $N > 2(M+2d)$.
Combined with \eqref{homogeneous1} and \eqref{penultimastima} this
shows that $(x_0,\xi_0) \notin \HWF(u)$, which completes the proof of
$\HWF(u) \subseteq \WF(u)$.
\end{proof}

\section{Existence of a tempered distribution with assigned Gabor wave front set}
\label{sec:exist}

For any given closed conic set $\Gamma \subseteq \rr {2d} \setminus \{0\}$, we construct in this section a distribution $u \in \cS'(\rr d)$ with $\WF(u)=\Gamma$. In order to do this we modify the construction given in \cite[Section~2.2]{Coriasco1} of a tempered distribution with assigned $\cS$ wave front set (cf. \cite{Coriasco2}), which itself was based on \cite[Theorem~8.1.4]{Hormander0}.

Constructing a distribution with given Gabor wave front set using pseudodifferential operators seems difficult. 
We construct $u$ in terms of modulated and translated Gaussian functions, which permits us to read off the decay properties of the STFT of $u$ from explicit expressions. This is an instance of the convenience of the characterization of $\WF(u)$ in terms of the STFT.

Let $(y,\eta) \in \rr {2d} \setminus \{ 0 \}$ and $k \in \no$ be fixed. We use the Gaussian window $\phi(x) = \pi^{-d/2} e^{-|x|^2/2}$ and set
\begin{equation}\label{eq:fk}
f_k(x;y,\eta) :=\pi^{d/2} M_{k^2 \eta} T_{k^2 y}\phi(x) = \exp\left(-\frac{1}{2}|x-k^2 y|^2 + i k^2 x \cdot \eta \right) \in \cS(\rr d). 
\end{equation} The modulus of the STFT of $f_k(\cdot,y,\eta)$ is
\begin{equation}\label{fkSTFT}
\big|V_\phi \big(f_k (\cdot;y,\eta)\big) (x,\xi) \big|
= \exp \left(-\frac{1}{4}\left(|x-k^2 y |^2 + |\xi - k^2 \eta|^2 \right) \right).
\end{equation}
For $g \in \cS(\rr d)$ we have
\begin{equation}\label{fkaction}
(f_k,g) = \pi^{d/2} \overline{V_\phi g( k^2 y,k^2 \eta)}.
\end{equation}
Since $V_\phi g \in \cS(\rr {2d})$ and $(y,\eta) \neq 0$, the series
\begin{equation}\label{fdefinition}
f(\cdot; y,\eta) = \sum_{k=1}^{\infty} \, f_k(\cdot;y,\eta)
\end{equation}
converges in $\cS'(\rr d)$ which gives $f(\cdot; y,\eta) \in \cS'(\rr d)$.
If $y \neq 0$ we have uniform convergence on compact subsets of $\rr d$, and $f(\cdot; y,\eta) \in L^\infty \cap C^\infty$.

\begin{thm}
For any closed conic set $\Gamma \subseteq \rr {2d}\setminus\{0\}$ there exists $u \in \cS'(\rr d)$ such that $\WF(u)=\Gamma$.
\label{thm:existence}
\end{thm}

\begin{proof}
Take a dense sequence of distinct vectors $\{w_j\}_{j \in \no} \subseteq \Gamma \cap S_{2d-1}$.
Then define by means of \eqref{fdefinition}
\begin{equation}\label{eq:Tdef}
u(x) := \sum_{j=0}^{\infty} 2^{-j} f(x;w_j), \quad x \in \rr d.
\end{equation}
It holds $u \in \cS'(\rr d)$. In fact, if $g \in \cS(\rr d)$ then by \eqref{fdefinition} and \eqref{fkaction}
\begin{align*}
|( f(\cdot;w_j), g) |
\leq \sum_{k=1}^\infty |(f_k(\cdot; w_j),g)|
= \pi^{d/2} \sum_{k=1}^\infty |V_\phi g( k^2 w_j)|.
\end{align*}
From $V_\phi g \in \cS(\rr {2d})$ and $|w_j|=1$ it follows that $|( f(\cdot;w_j), g) |$ is bounded by a constant uniformly over $j \in \no$, which shows that the sum in \eqref{eq:Tdef} converges in $\cS'(\rr d)$.

Suppose $0 \neq z_0 \notin \Gamma$.
By a standard scaling argument for separated cones, there exist an open conic set $\Gamma_0 \subseteq \rr {2d} \setminus \{0\}$ and $\ep>0$ such that $z_0 \in \Gamma_0$, and for any $z \in \Gamma_0$, $k \geq 1$, $w \in \Gamma$ with $|w|=1$, we have
\begin{align*}
|z - k^2w | & \geq \ep ( |z| + k^2).
\end{align*}
Using \eqref{fkSTFT} this gives for $z \in \Gamma_0$ and $N \geq 0$ arbitrary
\begin{align*}
\eabs{z}^N |V_\phi u(z)|
& \leq \sum_{j=0}^{\infty} 2^{-j} \sum_{k=1}^{\infty} \eabs{z}^N |V_\phi f_k(\cdot;w_j) (z)| \\
& \lesssim \sum_{k=1}^{\infty} \eabs{z}^N \exp \left(-\frac{\ep^2}{4} (|z|^2+k^4) \right)\lesssim 1.
\end{align*}
Thus $z_0 \notin \WF(u)$ which proves $\WF(u) \subseteq \Gamma$.

On the other hand, let $m \in \no$ be fixed and let $n \geq 1$ be an integer.
Formula \eqref{fkSTFT} gives
\begin{equation}\label{lowerbound1}
|V_\phi (f(\cdot; w_m))(n^2 w_m)|
\geq 1 - \sum_{k=1, \ k \neq n}^{\infty} |V_\phi (f_k(\cdot; w_m)) (n^2 w_m)|.
\end{equation}
We have again by \eqref{fkSTFT}
\begin{equation}\label{estimate1}
\begin{aligned}
|V_\phi (f_k(\cdot; w_j)) (n^2 w_m)|
& = \exp \left(-\frac{1}{4} \left| n^2 w_m - k^2 w_j \right|^2 \right) \\
& \leq \exp \left(-\frac{1}{4} (n+k)^2(n-k)^2 \right), \quad j \in \no.
\end{aligned}
\end{equation}
This gives
\begin{equation}
\label{zerolimit}
\begin{aligned}
\sum_{k\neq n} |V_\phi (f_k(\cdot; w_j)) (n^2 w_m)|
& \longrightarrow 0, \quad n \rightarrow +\infty,
\end{aligned}
\end{equation}
independently of $j \in \no$.

As a consequence, inserting $j=m$, we obtain from \eqref{lowerbound1}
\begin{align*}
|V_\phi (f(\cdot; w_m))(n^2 w_m)| \geq \frac1{2}, \quad n \geq N,
\end{align*}
for some integer $N \geq 1$, which in turn yields for $n \geq N$
\begin{equation}\label{lowerbound2}
\begin{aligned}
|V_\phi u(n^2 w_m)|
& = \left| \sum_{j=0}^{\infty} 2^{-j} V_\phi (f(\cdot;w_j) ) (n^2 w_m) \right| \\
& \geq 2^{-m-1} - \sum_{j=0, \, j \neq m}^{\infty} 2^{-j} \left| V_\phi (f(\cdot;w_j) ) (n^2 w_m) \right|.
\end{aligned}
\end{equation}

To estimate the remainder, we obtain from \eqref{fdefinition}, \eqref{estimate1} and \eqref{zerolimit}
\begin{align*}
\left| V_\phi (f(\cdot;w_j) ) (n^2 w_m) \right| & \leq 2, \quad j \in \no, \quad n \geq N,
\end{align*}
after possibly increasing $N$. This uniform bound with respect to $j \in \no$ implies that there exists
$M \geq m+1$ such that
\begin{equation}\label{remainder1}
\sum_{j=M}^{\infty} 2^{-j} \left| V_\phi (f(\cdot;w_j) ) (n^2 w_m) \right| \leq 2^{-m-3}, \quad n \geq N.
\end{equation}

For $0 \leq j \leq M-1$ and $j \neq m$ we have $|w_m-w_j| \geq \delta$ for some $\delta>0$.
This observation combined with \eqref{fkSTFT}, \eqref{fdefinition} and \eqref{zerolimit} gives
\begin{align*}
& \left| V_\phi (f(\cdot;w_j) ) (n^2 w_m) \right| \\
& \leq \left| V_\phi (f_n(\cdot;w_j) ) (n^2 w_m) \right|
+ \sum_{k\neq n}\left| V_\phi (f_k(\cdot;w_j) ) (n^2 w_m) \right|\\
& \leq \exp \left(-n^4 \delta^2/4 \right) + 2^{-m-5}
 \leq 2^{-m-4}, \quad 0 \leq j \leq M-1, \quad n \geq N,
\end{align*}
again after possibly increasing $N$.
We may conclude
\begin{align*}
\sum_{j=0, \, j \neq m}^{M-1} 2^{-j} \left| V_\phi (f(\cdot;w_j) ) (n^2 w_m) \right| \leq 2^{-m-3}, \quad n \geq N,
\end{align*}
which together with \eqref{remainder1} inserted into \eqref{lowerbound2} shows that
$$
|V_\phi u(n^2 w_m)| \geq 2^{-m-2}
$$
for $n \geq N$.
Thus $V_\phi u$ does not decay rapidly in any conic neighbourhood of $w_m$,
and it follows that $w_m \in \WF(u)$. This holds for all $m \in \no$, and on account of $\WF(u)$ being closed in $\rr {2d} \setminus \{ 0 \}$ it follows that $\Gamma \subseteq \WF(u)$.
\end{proof}

\appendix

\section{Global semiclassical pseudodifferential calculus}\label{appendix}

Here we develop a pseudodifferential calculus suitable for analyzing the expressions arising in the study of the homogeneous wave front set.
First we introduce a symbol class where the functions depend on a semiclassical small parameter $h$
and behave like dilated Shubin symbols (cf. \cite{Nicola1,Shubin1}).

\begin{defn}
Let $a=a(z;h) \in C^\infty (\rr {2d} \times (0,1])$. Then $a$ is a $h$-\emph{Shubin symbol} of order $m \in \ro$, denoted $a\in G_h^m$, if for all $\alpha \in \nn {2d}$ there exists a constant $C_{\alpha}>0$ such that
\begin{equation}\label{shubinclassh}
	|\pdd z \alpha a(z;h)| \leq C_{\alpha} h^{|\alpha|} \eabs{hz}^{m-|\alpha|}, \quad z \in \rr {2d}, \quad h \in (0,1].
\end{equation}
\end{defn}

The space $G_h^m$ is a Fr\'echet space equipped with the topology defined by the seminorms
\begin{equation*}
\sup_{z \in \rr {2d}, \ h \in (0,1]} h^{-|\alpha|} \eabs{h z}^{|\alpha|-m} |\pdd z \alpha a(z;h)|, \quad \alpha \in \nn {2d}.
\end{equation*}
A typical case of an $h$-dependent symbol in $G_h^m$ is
$$
a_h(z) = a(hz), \quad z \in \rr {2d}, \quad h \in (0,1],
$$
where $a \in G^m$.
The condition \eqref{shubinclassh} may be described equivalently as $a(\cdot/h;h) \in G^m$ uniformly over $h \in (0,1]$.
We observe that $a\in G_h^m$ and $b\in G_h^n$ implies $ab \in G_h^{m+n}$ and that $\partial^\alpha_z$ maps $G_h^{m}$ into $h^{|\alpha|} G_h^{m-|\alpha|}$.

\subsection{Asymptotic expansions}

\begin{defn}\label{asymptoticexpansion1}
Let $a_j \in G_h^{m_j}$, $j=0,1,\dots$, where $m_j \rightarrow - \infty$ as $j \rightarrow + \infty$, let $a \in C^\infty(\rr{2d} \times (0,1])$ and $h \in (0,1]$.
Set $\wt m_j = \max_{k \geq j} m_k$.
We write
$$
a \sim \sum_{j=0}^\infty h^{\wt m_0-m_j} a_j
$$
provided
\begin{equation}\label{restterm1}
a-\sum_{j=0}^{N-1} h^{\wt m_0-m_j} a_j \in h^{\wt m_0-\wt m_N} G_h^{\wt m_N}, \quad N \geq 1,
\end{equation}
i.e. for any $\alpha \in \nn {2d}$ there exists $C_{N,\alpha}>0$ such that
\begin{align*}
\left| \pdd z \alpha \left( a(z;h)-\sum_{j=0}^{N-1} h^{\wt m_0-m_j} a_j(z;h) \right) \right|
& \leq C_{N,\alpha} h^{\wt m_0-\wt m_N+|\alpha|} \eabs{h z}^{\wt m_N - |\alpha|}, \\
& \qquad z \in \rr {2d}, \quad h \in (0,1].
\end{align*}
\end{defn}

We note that $a \sim 0$ if and only if $a \in \cap_{k \in \no} h^k G_h^{-k}$.

In the following lemma we prove that series of symbols with orders approaching $-\infty$, jointly vanishing around zero when dilated by $1/h$, may be asymptotically summed up.

\begin{lem}\label{asymptotic2}
Suppose $a_j \in G_h^{m_j}$ for $j=0,1,\dots$, where $m_j \rightarrow - \infty$ as $j \rightarrow \infty$, and suppose there exists $\ep>0$ such that
\begin{equation}\label{supportassumption}
\supp(a_j (\cdot/h;h) )\cap B_\ep  = \emptyset, \quad h \in (0,1], \quad j=0,1,\dots .
\end{equation}
Then there exists $a \in G_h^{\wt m_0}$ such that $a \sim \sum_0^\infty h^{\wt m_0 - m_j} a_j$.
\end{lem}

\begin{proof}
First we observe that the requirement \eqref{restterm1} is invariant to a reordering of the indices.
By summing symbols of equal order we may thus assume that $m_0 > m_1> \cdots$.

Let $0 \leq \chi \in C_c^\infty(\rr {2d})$ satisfy $\chi(z)=1$ for $|z| \leq 1$ and $\chi(z)=0$ for $|z| \geq 2$,
and let $(\ep_j)$ be a sequence of positive numbers such that $\ep_j \rightarrow 0$ as $j \rightarrow +\infty$.
Set for $0 < h \leq 1$
\begin{equation*}
a(z;h) = \sum_{j=0}^\infty h^{m_0-m_j} a_j (z;h) (1-\chi(\ep_j z)).
\end{equation*}
For $z \in \rr {2d}$ and $h \in (0,1]$ fixed, in a neighborhood of $z$ the sum contains only a finite number of terms since $\chi(\ep_j z)=1$ for $\ep_j |z| \leq 1$. Hence $a \in C^\infty(\rr{2d} \times (0,1])$.

In order to show $a \sim \sum_0^\infty h^{m_0-m_j} a_j$ it suffices to show for any $N \geq 1$ and any $\alpha \in \nn {2d}$ the estimate for $z \in \rr {2d}$ and $h \in (0,1]$
\begin{equation}\label{sufficientestimate1}
\left| \pdd z \alpha \left(  a(z/h;h)-\sum_{j=0}^{N-1} h^{m_0-m_j} a_j(z/h;h) \right) \right|
\leq C_{N,\alpha} h^{m_0-m_N} \eabs{z}^{m_N-|\alpha|}.
\end{equation}
We have
\begin{equation}\label{threeterms}
\begin{aligned}
a(z/h;h)-\sum_{j=0}^{N-1} &h^{m_0-m_j} a_j(z/h;h)\\
& = - \sum_{j=0}^{N-1} h^{m_0-m_j} a_j(z/h;h) \, \chi(\ep_j z/h)\\
&\quad+ \sum_{j=N+|\alpha|+1}^\infty h^{m_0-m_j} a_j(z/h;h) \, (1-\chi(\ep_j z/h))\\
&\quad+ \sum_{j=N}^{N+|\alpha|} h^{m_0-m_j} a_j(z/h;h) \, (1-\chi(\ep_j z/h)).
\end{aligned}
\end{equation}

First we look at
\begin{align*}
\pdd z \alpha \big( a_j(z/h;h) & \, (1-\chi(\ep_j z/h)) \big)
= \pdd z \alpha  (a_j(z/h;h) ) (1-\chi(\ep_j z/h)) \\
& - \sum_{0 < \beta \leq \alpha} \binom{\alpha}{\beta} \partial_z^{\alpha-\beta}  (a_j(z/h;h) )(\ep_j/ h)^{|\beta|} \pdd z \beta \chi (\ep_j z/h)
\end{align*}
when $j \geq N+|\alpha|+1$.
In the support of $1-\chi(\ep_j z/h)$ we have $|z| \geq h \ep_j^{-1}$ so $m_j<m_N$ gives
\begin{align*}
\left| \pdd z \alpha  (a_j(z/h;h)) \, (1-\chi(\ep_j z/h)) \right|
& \leq C_{j,\alpha} \eabs{z}^{m_N-|\alpha|} \eabs{z}^{m_j-m_N} (1-\chi(\ep_j z/h)) \\
& \leq C_{j,\alpha} \, \ep_j^{m_N-m_j}  h^{m_j-m_N}  \eabs{z}^{m_N-|\alpha|} \\
& \leq 2^{-j-|\alpha|} \, h^{m_j-m_N}  \eabs{z}^{m_N-|\alpha|}, \quad j \geq N+|\alpha|+1,
\end{align*}
provided $\ep_j>0$ is chosen so small that
$C_{j,\alpha} \, \ep_j^{m_N-m_j} \leq 2^{-j-|\alpha|}$ for all $N,\alpha$ such that
$N+|\alpha|+1 \leq j$.  The choice of $\ep_j$ thus depends only on $j$.
The assumption \eqref{supportassumption} implies that we have $\eabs{z} \leq C |z|$ for some $C>0$ in the support of $a_j(\cdot/h;h)$ for all $h \in (0,1]$ and all $j \in \no$.
This gives for $0 < \beta \leq \alpha$,
using the fact that on the support of $\pd \beta \chi(\ep_j z/h)$ we have $|z| \geq h \ep_j^{-1}$,
\begin{align*}
& \big| \partial_z^{\alpha-\beta}  (a_j(z/h;h) ) (\ep_j /h)^{|\beta|} \pd \beta \chi (\ep_j z/h) \big| \\
& \leq C_{j,\alpha} \eabs{z}^{m_N-|\alpha|} \eabs{z}^{m_j-m_N+|\beta|}
(\ep_j /h)^{|\beta|} \left| \pd \beta \chi (\ep_j z/h) \right| \\
& \leq C_{j,\alpha} \, \ep_j^{m_N-m_j}  \, h^{m_j-m_N} \eabs{z}^{m_N-|\alpha|}
\left|\ep_j  z/h \right|^{|\beta|} \left| \pd \beta \chi (\ep_j z/h) \right| \\
& \leq C_{j,\alpha} \, \ep_j^{m_N-m_j}  h^{m_j-m_N} \eabs{z}^{m_N-|\alpha|} \\
& \leq 2^{-j-|\alpha|} h^{m_j-m_N} \eabs{z}^{m_N-|\alpha|}, \quad j \geq N+|\alpha|+1,
\end{align*}
provided $\ep_j>0$ is sufficiently small.
Combining the latter two estimates yields for $z \in \rr {2d}$
\begin{equation}\label{asymptotic1}
\left| \pdd z \alpha \big( a_j(z/h;h) (1-\chi(\ep_j z/h)) \big) \right|
\leq 2^{-j} h^{m_j-m_N} \eabs{z}^{m_N-|\alpha|}, \quad j \geq N+|\alpha|+1.
\end{equation}

We estimate the $\alpha$th derivative of the second term of the right hand side of \eqref{threeterms} by means of \eqref{asymptotic1} as, for $z \in \rr {2d}$,
\begin{equation}\label{secondtermestimate}
\begin{aligned}
& \left| \sum_{j=N+|\alpha|+1}^\infty h^{m_0-m_j} \pdd z \alpha \left( a_j(z/h;h) (1-\chi(\ep_j z/h)) \right) \right| \\
& \leq h^{m_0-m_N} \eabs{z}^{m_N-|\alpha|} \sum_{j=N+|\alpha|+1}^\infty 2^{-j} \\
& = 2^{-N-|\alpha|} h^{m_0-m_N} \eabs{z}^{m_N-|\alpha|}.
\end{aligned}
\end{equation}

Next we estimate the $\alpha$th derivative of the first term of \eqref{threeterms}, again using $\eabs{z} \leq C |z|$:
\begin{equation}\label{firsttermestimate}
\begin{aligned}
& \left| \sum_{j=0}^{N-1} h^{m_0-m_j} \pdd z \alpha \left( a_j(z/h;h) \chi(\ep_j z/h) \right) \right| \\
& \leq \sum_{j=0}^{N-1} h^{m_0-m_j} \sum_{\beta \leq \alpha} \binom{\alpha}{\beta} \left| \partial_z^{\alpha-\beta} (a_j(z/h;h) ) \right| \,
\left| (\ep_j /h)^{|\beta|} \pd \beta \chi(\ep_j z/h) \right| \\
& \leq h^{m_0-m_N} \eabs{z}^{m_N-|\alpha|} \sum_{j=0}^{N-1} \sum_{\beta \leq \alpha} C_{j,\beta} \, \ep_j^{m_N-m_j}
\left| z \ep_j /h \right|^{m_j-m_N+|\beta|}
\left|  \pd \beta \chi(\ep_j z/h) \right| \\
& \leq C_{N,\alpha} h^{m_0-m_N} \eabs{z}^{m_N-|\alpha|}, \quad z \in \rr {2d}.
\end{aligned}
\end{equation}

Finally we estimate the $\alpha$th derivative of the third term of \eqref{threeterms} as
\begin{equation}\label{thirdtermestimate}
\begin{aligned}
& \left| \sum_{j=N}^{N+|\alpha|} h^{m_0-m_j} \pdd z \alpha \Big( a_j(z/h;h) \, (1-\chi(\ep_j z/h)) \Big) \right| \\
& \leq h^{m_0-m_N} \sum_{j=N}^{N+|\alpha|} \sum_{0 \leq \beta \leq \alpha} \binom{\alpha}{\beta} \left| \partial_z^{\alpha-\beta} a_j(z/h;h) \right| \, \left| \pdd z \beta (1-\chi(\ep_j z/h) ) \right| \\
& \leq h^{m_0-m_N}  \eabs{z}^{m_N - |\alpha|} \sum_{j=N}^{N+|\alpha|} \left( C_{j,\alpha} + \sum_{0 < \beta \leq \alpha} C_\beta \left|\ep_j z /h \right|^{|\beta|} \left| \pd \beta \chi (\ep_j z/h) \right| \right) \\
& \leq C_{N,\alpha} h^{m_0-m_N}  \eabs{z}^{m_N-|\alpha|}, \quad z \in \rr {2d}.
\end{aligned}
\end{equation}

A combination of the estimates \eqref{secondtermestimate}, \eqref{firsttermestimate} and \eqref{thirdtermestimate} shows that there exists $C_{N,\alpha}>0$ such that \eqref{sufficientestimate1} holds.
This proves $a \sim \sum_0^\infty h^{m_0-m_j} a_j$, and $a \in G_h^{m_0}$ follows.
\end{proof}

\begin{prop}\label{regularize1}
If $u \in \cS'(\rr d)$ and $a \sim 0$ then
\begin{equation*}
\| a^w(x,D) u\|_{L^2} = \cO(h^\infty).
\end{equation*}
\end{prop}

\begin{proof}
By the assumption $u \in \cS'(\rr d)$ and \cite[Corollary~25.2]{Shubin1}, there exists $s \in \ro$ such that $u \in Q^s(\rr d)$.
Here $Q^s(\rr d)$ denotes a Sobolev type space defined by the norm
\begin{equation*}
\| u \|_{Q^s} = \| A_s u \|_{L^2}
\end{equation*}
where $A_s$ is the localization (or anti-Wick) operator (cf. \cite{Nicola1,Shubin1}), defined by
\begin{equation*}
(A_s u,g) = (2 \pi)^{-d} (v_s V_\fy u, V_\fy g), \quad g \in \cS(\rr d), \quad u \in \cS'(\rr d),
\end{equation*}
where $\fy \in \cS$ and $\| \fy \|_{L^2}=1$, with symbol $v_s (z)=\eabs{z}^s$, $z \in \rr {2d}$ (cf. \cite[Definition~1.5.2 and Proposition~1.7.12]{Nicola1}). The Sobolev--Shubin space $Q^s(\rr d)$ equals the modulation space $M_{v_s}^2(\rr d)$, see \cite{Boggiatto1},  \cite{Grochenig1} and \cite[Corollary~1.7.17]{Nicola1}. 

By assumption we have
$a \in \cap_{N \geq 0} h^N G_h^{-N}$, which implies $a \in \cap_{N \geq 0} h^N G^{-N}$.
In fact, let $N \in \no$. For $|z| \leq 1$ we have
\begin{equation*}
\eabs{z}^{N +|\alpha|} |\pd \alpha a(z;h)| \leq C_{N,\alpha} h^{N+|\alpha|} \eabs{ h z}^{-N-|\alpha|}
\leq C_{N,\alpha} h^N,
\end{equation*}
and for $|z| \geq 1$ we have, using $\eabs{hz}^2 \geq 2 h |z|$,
\begin{align*}
|\pd \alpha a(z;h)|
& \leq C_{N,\alpha} h^{2N+2|\alpha|} \eabs{ h z}^{-2N-2|\alpha|} \\
& \leq C_{N,\alpha} h^{N+|\alpha|} |z|^{-N-|\alpha|} \\
& \leq C_{N,\alpha} h^{N} \eabs{z}^{-N-|\alpha|}.
\end{align*}

Thus for any $N \in \no$ we have $a=h^N b$ where $b \in G^{-N}$.
If we pick $N \geq -s$ then by
\cite[Proposition~25.4]{Shubin1} (cf. \cite[Proposition~1.5.5 and Theorem~2.1.10]{Nicola1}), since $Q^0(\rr d)=L^2(\rr d)$,
\begin{equation*}
\| b^w(x,D) u \|_{L^2} \leq C \| u \|_{Q^s} < \infty,
\end{equation*}
which implies
\begin{equation*}
\| a^w(x,D) u \|_{L^2} \lesssim h^N, \quad h \in (0,1], \quad N \in \no.
\end{equation*}
\end{proof}

\subsection{Weyl product and change of quantization}

In the next theorem we show the asymptotic expansion formula for the Weyl product for the calculus
of $h$-Shubin symbols.
Here we use the notation $D=(D_y, D_\eta; D_t,D_\theta)$ for $(-i)$ times the gradient with respect to $(y,\eta; t,\theta) \in \rr {4d}$, and
\begin{equation*}
\sigma(D) = D_\eta \cdot D_t - D_y \cdot D_\theta
= \sum_{k=1}^d \partial_{y_k} \partial_{\theta_k} - \partial_{\eta_k} \partial_{t_k}
\end{equation*}
for the symplectic form $\sigma$ with gradient arguments.

\begin{thm}\label{calculuscomp0}
If $a \in G_h^m$ and $b \in G_h^n$ then $a \wpr b \in G_h^{m+n}$ and
\begin{equation}\label{calculuscomp1}
a \wpr b(z)  \sim \sum_{j=0}^{\infty} h^{2j} \frac{\left( i \sigma(D)/2 \right)^j}{j!} \, (a \otimes b)_{h^{-1}} \big|_{h(z,z)}.
\end{equation}
\end{thm}

\begin{proof}
For $a,b \in \cS(\rr {2d})$ we have by \cite[Eq.~(18.5.6)]{Hormander0} and a change of variables
\begin{equation}\label{weylproduct1}
a \wpr b (z) = \exp \left( \frac{i}{2} h^{2} \sigma(D) \right) (a \otimes b)_{h^{-1}} \big|_{h(z,z)}
\end{equation}
where $\exp( i h^{2} \sigma(D)/2)$ is defined as the Fourier multiplier operator with symbol $\exp( i h^{2} \sigma/2)$.

First we prove the estimate for $a \in G^m$, $b \in G^n$ and $N \geq \max(m,n)+2d+5$
\begin{equation}\label{restestimate1}
\begin{aligned}
& \left| \left( \exp \left( \frac{i}{2} h^{2} \sigma(D) \right)  - \sum_{j=0}^{N-1} \frac{\left(i h^2 \sigma(D)/2 \right)^j}{j!} \,   \right) (a \otimes b) (z,w)\right| \\
& \leq C_{N} h^{2N} \eabs{z}^{m+2d+5-N} \eabs{w}^{n+2d+5-N}, \quad z,w \in \rr {2d}.
\end{aligned}
\end{equation}
To show \eqref{restestimate1} we start with $a,b \in \cS(\rr {2d})$.
From the Taylor expansion, $Y \in \rr {4d}$,
\begin{align*}
e^{i h^{2} \sigma(Y)/2} & =  \sum_{j=0}^{N-1} \frac{\left(i h^2 \sigma(Y) /2 \right)^j}{j!} \\
& \qquad + \frac1{(N-1)!} \int_0^1 (1-t)^{N-1} e^{ i t h^{2} \sigma(Y)/2}  \left(i h^2 \sigma(Y) /2\right)^N dt,
\end{align*}
we obtain
\begin{equation}\label{restestimate2}
\begin{aligned}
& \left( \exp \left( \frac{i}{2} h^{2} \sigma(D) \right)  - \sum_{j=0}^{N-1} \frac{\left( i h^2 \sigma(D)/2 \right)^j}{j!} \,   \right) (a \otimes b) (X) \\
& = \frac1{(N-1)!} \int_0^1 (1-t)^{N-1} e^{ i t h^{2} \sigma(D)/2}  \left(i h^2 \sigma(D) /2\right)^N (a \otimes b) (X) \, dt, \quad X \in \rr {4d}.
\end{aligned}
\end{equation}

Next we use
\begin{equation}\label{japanese1}
\eabs{z}^{2k} \lesssim \max_{|\alpha| \leq 2 k} |z^\alpha|, \quad z \in \rr {2d}, \quad k \in \no,
\end{equation}
\begin{equation}\label{japanese2}
\| f \|_{L^1(\rr {2d})} = \| \eabs{\cdot}^{-2d-2} \eabs{\cdot}^{2d+2} f \|_{L^1(\rr {2d})}
\lesssim \| \eabs{\cdot}^{2d+2} f \|_{L^\infty(\rr {2d})},
\end{equation}
and, since $N \geq \max(m,n)+2d+5$, the existence of nonnegative integers $M$ and $L$ such that
\begin{equation}\label{exponent1}
\begin{aligned}
N-m-2d-5 & \leq 2 M \leq N-m-2d-3, \\
N-n-2d-5 & \leq 2 L \leq N-n-2d-3.
\end{aligned}
\end{equation}

This gives the following estimate for $z,w \in \rr {2d}$, writing the variable of $a \otimes b$ as $X=(x_1,x_2)$, $x_1,x_2 \in \rr {2d}$,
\begin{align*}
& \eabs{z}^{2M} \eabs{w}^{2L} | e^{ i t h^{2} \sigma(D)/2} \sigma(D)^N (a \otimes b) (z,w)| \\
& \stackrel{\eqref{japanese1}}{\lesssim} \max_{|\alpha_1| \leq 2 M, \  |\alpha_2| \leq 2 L} \left| z^{\alpha_1} w^{\alpha_2}  \int e^{i (z,w) \cdot Y} e^{ i t h^{2} \sigma(Y)/2} \cF \left( \sigma(D)^N (a \otimes b) \right) (Y) dY \right| \\
& \lesssim \max_{\stackrel{|\alpha_1| \leq 2 M, \  |\alpha_2| \leq 2 L}{\small \beta_1 \leq \alpha_1, \ \beta_2 \leq \alpha_2}} \left| \int e^{i (z,w) \cdot Y} (-D_Y)^{\alpha_1-\beta_1,\alpha_2-\beta_2} (e^{ i t h^{2} \sigma(Y)/2}) \right. \\
& \left. \qquad \qquad \qquad \qquad \qquad \qquad \qquad \times \cF \left( x_1^{\beta_1} x_2^{\beta_2} \sigma(D)^N (a \otimes b) \right) (Y) \, dY \right|.
\end{align*}

Note that $(-D_Y)^{\alpha_1-\beta_1,\alpha_2-\beta_2} (e^{ i t h^{2} \sigma(Y)/2})=p(Y) (e^{ i t h^{2} \sigma(Y)/2})$ where $p$ is a polynomial such that $\deg(p) \leq 2(M+L)$ with coefficients that contain powers of $t h^2$. These powers are uniformly bounded by one when $t \in [0,1]$ and $h \in (0,1]$.
Combined with the facts $\eabs{\cdot}^{-4d-2}, \eabs{\cdot}^{-2d-2} \otimes \eabs{\cdot}^{-2d-2} \in L^1(\rr {4d})$ this gives
\begin{equation}\label{restestimate2b}
\begin{aligned}
& \sup_{z,w \in \rr {2d}} \eabs{z}^{2 M} \eabs{w}^{2 L} | e^{ i t h^{2} \sigma(D)/2} \sigma(D)^N (a \otimes b) (z,w)| \\
& \lesssim \sup_{\stackrel{|\beta_1| \leq 2M, \  |\beta_2| \leq 2 L}{\small |\gamma| \leq 2(M+L+2d+1), \ Y \in \rr {4d}}} \left| \cF (D^{\gamma} (x_1^{\beta_1} x_2^{\beta_2} \sigma(D)^N (a \otimes b))) (Y) \right| \\
& \leq \sup_{\stackrel{|\beta_1| \leq 2M, \  |\beta_2| \leq 2 L}{\small |\gamma| \leq 2(M+L+2d+1)}} \left\| D^{\gamma} (x_1^{\beta_1} x_2^{\beta_2} \sigma(D)^N (a \otimes b))) \right\|_{L^1(\rr {4d})} \\
& \stackrel{\eqref{japanese2}}{\lesssim} \sup_{\stackrel{|\beta_1| \leq 2M, \  |\beta_2| \leq 2 L, \ x_1,x_2 \in \rr {2d}}{\small |\gamma| \leq 2(M+L+2d+1)}} \eabs{x_1}^{2d+2} \eabs{x_2}^{2d+2} \left| D^{\gamma} (x_1^{\beta_1} x_2^{\beta_2} \sigma(D)^N (a \otimes b)(x_1,x_2)) \right|.
\end{aligned}
\end{equation}

In the following we use the notation
\begin{equation*}
\| a \|_{m,k} = \sup_{z \in \rr {2d}, \ |\gamma|=k} \eabs{z}^{k - m} | \pd \gamma a(z)|, \quad k \in \no,
\end{equation*}
for the seminorms of $G^m$.
Then \eqref{restestimate2b} can be estimated as
\begin{align*}
& \sup_{z,w \in \rr {2d}} \eabs{z}^{2 M} \eabs{w}^{2 L} | e^{ i t h^{2} \sigma(D)/2} \sigma(D)^N (a \otimes b) (z,w)| \\
& \lesssim \sup_{x_1,x_2 \in \rr {2d}} \eabs{x_1}^{2d+2+2M} \eabs{x_2}^{2d+2+2L} \sum_{N \leq |\alpha_1|, \, |\alpha_2| \leq N + 2(M+L+2d+1)} |\partial^{\alpha_1} a(x_1) \partial^{\alpha_2} b(x_2) | \\
& \lesssim \sup_{x_1,x_2 \in \rr {2d}} \eabs{x_1}^{2d+2+2M+m+1-N} \eabs{x_2}^{2d+2+2L+n+1-N} \\
& \qquad \qquad \times \max_{N \leq k \leq N + 2(M+L+2d+1)} \| a \|_{m+1,k} \max_{N \leq k \leq N + 2(M+L+2d+1)} \| b \|_{n+1,k} \\
& \leq \max_{N \leq k \leq N + 2(M+L+2d+1)} \| a \|_{m+1,k} \max_{N \leq k \leq N + 2(M+L+2d+1)} \| b \|_{n+1,k}.
\end{align*}

Inserted into \eqref{restestimate2} this yields the estimate for $a,b \in \cS(\rr {2d})$ and $N \geq \max(m,n)+2d+5$
\begin{equation}\label{restestimate3}
\begin{aligned}
& \left| \left( \exp \left( \frac{i}{2} h^{2} \sigma(D) \right)  - \sum_{j=0}^{N-1} \frac{\left(i h^2 \sigma(D)/2 \right)^j}{j!} \,   \right) (a \otimes b) (z,w) \right| \\
& \lesssim C_N h^{2N} \eabs{z}^{m+2d+5-N}  \eabs{w}^{n+2d+5-N} \\
& \qquad \qquad \times \max_{N \leq k \leq N + 2(M+L+2d+1)} \| a \|_{m+1,k} \max_{N \leq k \leq N + 2(M+L+2d+1)} \| b \|_{n+1,k}.
\end{aligned}
\end{equation}

In the next step of the proof we let $a \in G^m$, $b \in G^n$ and approximate $a$ and $b$ as $a_\ep(z)=a(z) \chi(\ep z)$ and $b_\ep(z)=b(z) \chi(\ep z)$ where $\chi \in C_c^\infty(\rr {2d})$, $\chi(z)=1$ for $|z| \leq 1$, $\chi(z)=0$ for $|z| \geq 2$ and $\ep>0$. Then $a_\ep, b_\ep \in \cS(\rr {2d})$, $a_\ep \rightarrow a$ in $G^{m+1}$ and $b_\ep \rightarrow b$ in $G^{n+1}$ as $\ep \rightarrow 0$.
It follows from \eqref{restestimate3} that
for any $z,w \in \rr{2d}$
\begin{equation}\label{pointwise1}
\exp \left( \frac{i}{2} h^{2} \sigma(D) \right) (a \otimes b)(z,w)
= \lim_{\ep \rightarrow 0} \exp \left( \frac{i}{2} h^{2} \sigma(D) \right) (a_\ep \otimes b_\ep)(z,w).
\end{equation}
Furthermore it follows that \eqref{restestimate1} holds for $a \in G^m$, $b \in G^n$ and $N \geq \max(m,n)+2d+5$.

Let $a \in G_h^m$ and $b \in G_h^n$.
With $\chi \in C_c^\infty(\rr {2d})$ as above and
$a_\ep(z) = a(z) \chi(\ep h z)$ and $b_\ep(z)=b(z) \chi(\ep h z)$
we have $a_\ep \rightarrow a$ in $G_h^{m+1}$ and $b_\ep \rightarrow b$ in $G_h^{n+1}$ as $0<\ep \rightarrow 0$.
For $u \in \cS$ it follows that $a_\ep^w(x,D) u \rightarrow a^w(x,D) u$ in $\cS$, which implies
$$
a_\ep^w(x,D) \ b_\ep^w(x,D) u \rightarrow a^w(x,D) b^w(x,D) u \quad \mbox{in $\cS'$ as $\ep \rightarrow 0$}.
$$
Since $(a_\ep)_{h^{-1}} \rightarrow a_{h^{-1}}$ in $G^{m+1}$ and
$(b_\ep)_{h^{-1}} \rightarrow b_{h^{-1}}$ in $G^{n+1}$ we obtain from
\eqref{weylwigner}, \eqref{weylproduct1}, \eqref{pointwise1} and dominated convergence, for $u,v \in \cS(\rr d)$,
\begin{align*}
(a^w(x,D) \, b^w(x,D) \, u, v) & = (2 \pi)^{-d} \lim_{\ep \rightarrow 0} (a_\ep \wpr b_\ep ,W(v,u) ) \\
& = (2 \pi)^{-d} \lim_{\ep \rightarrow 0} \left( e^{i h^{2} \sigma(D)/2} (a_\ep \otimes b_\ep)_{h^{-1}} \big|_{h(\cdot,\cdot)} ,W(v,u) \right) \\
& = (2 \pi)^{-d} \left( e^{i h^{2} \sigma(D)/2} (a \otimes b)_{h^{-1}} \big|_{h(\cdot,\cdot)} ,W(v,u) \right).
\end{align*}
It follows that \eqref{weylproduct1} extends to $a \in G_h^m$ and $b \in G_h^n$.

Finally let $a \in G_h^m$ and $b \in G_h^n$, let $N \geq 1$ be arbitrary and let $\wt N \geq \max(m,n,N)+2d+5$.
Since derivatives commute with the operator $\exp \left( i h^{2} \sigma(D)/2 \right)$, \eqref{weylproduct1} and \eqref{restestimate1} yield the estimate
\begin{equation}\label{restestimate4}
\begin{aligned}
& \left| \pdd z \alpha \left( a \wpr b (z) - \sum_{j=0}^{\wt N-1} \frac{\left( i h^2\sigma(D)/2 \right)^j}{j!} \, (a \otimes b)_{h^{-1}} \big|_{h(z,z)} \right) \right| \\
& \leq C_{N,\alpha} h^{2 \wt N+|\alpha|} \eabs{h z}^{m+n+4d+10-2 \wt N-|\alpha|} \\
& \leq C_{N,\alpha} h^{2 N+|\alpha|} \eabs{h z}^{m+n-2N-|\alpha|}, \quad z \in \rr {2d}.
\end{aligned}
\end{equation}
We write
\begin{equation*}
\begin{aligned}
a \wpr b (z) - & \sum_{j=0}^{N-1} \frac{\left( i h^2 \sigma(D)/2 \right)^j}{j!} (a \otimes b)_{h^{-1}} \big|_{h(z,z)}  \\
& = a \wpr b (z) - \sum_{j=0}^{\wt N-1} \frac{\left( i h^2 \sigma(D)/2 \right)^j}{j!} (a \otimes b)_{h^{-1}} \big|_{h(z,z)}  \\
& + \sum_{j=N}^{\wt N-1} \frac{\left( i h^2 \sigma(D)/2 \right)^j}{j!} (a \otimes b)_{h^{-1}} \big|_{h(z,z)}
\end{aligned}
\end{equation*}
and observe that the second sum of the right hand side belongs to $h^{2N}G_h^{m+n-2N}$.
Combined with \eqref{restestimate4}  this proves \eqref{calculuscomp1}, and
it follows that $a \wpr b \in G_h^{m+n}$.
\end{proof}

The next result treats invariance of the symbol class $G_h^m$ with respect to a switch from the Kohn--Nirenberg  to the Weyl quantization,
and the corresponding asymptotic expansion.
Here we use the fact that for $a,b \in \cS'(\rr {2d})$ and $b^w(x,D)=a(x,D)$ we have
\begin{align*}
b(x,\xi) & = \exp\left(-\frac{i}{2} D_x \cdot D_\xi \right) a(x,\xi) \\
& = \exp\left(-\frac{i}{2} h^2 D_x \cdot D_\xi \right) (a)_{h^{-1}} \big|_{(hx,h\xi)}
\end{align*}
(cf. \cite[Chapter 18.5]{Hormander0}).
The proof is omitted since it is analogous to the proof of Theorem \ref{calculuscomp0}.

\begin{thm}\label{KNweyl}
If $a \in G_h^m$ is a Kohn--Nirenberg symbol then the corresponding Weyl symbol $b$
defined by $b^w(x,D)=a(x,D)$ satisfies $b \in G_h^m$ and
\begin{equation*}
b (x,\xi) \sim \sum_{j=0}^\infty h^{2j} \frac{\left( -\frac{i}{2} D_x \cdot D_\xi \right)^j}{j!} (a)_{h^{-1}} \big|_{(hx,h\xi)}.
\end{equation*}
\end{thm}

\subsection{An invariance result for the homogeneous wave front set}

The following lemma shows that a symbol in $G_h^0$, bounded from below outside an $1/h$-dilated neighborhood of the origin, admits a symbol acting as a parametrix with respect to a symbol compactly supported outside the neighborhood of the origin (in analogy with \cite[Theorem 2.3.3]{Cordes}).

\begin{lem}\label{parametrix1}
Suppose $a \in G_h^0$ satisfies $| a(z;h) | \geq C > 0$ for all $z \in \rr {2d}$
such that $h |z| \geq R>0$ and all $h \in (0,1]$,
and let $U \subseteq \rr {2d}$ be a relatively compact neighborhood such that
$\overline{U} \cap \overline{B_R} = \emptyset$.
Then there exists $c \in G_h^0$ such that for any $b \in C_c^\infty(U)$
$$
b_h \wpr c \wpr a - b_h \in \bigcap_{k \in \no} h^k G_h^{-k}.
$$
\end{lem}

\begin{proof}
Let $\chi \in C_c^\infty(\rr {2d})$ satisfy $0 \leq \chi \leq 1$, $\supp(\chi) \cap \overline{U} = \emptyset$ and $\chi(z)=1$ when $z \in B_{R'}$ for some $R'>R$.
Using the lower bound on $|a(z;h)|$, it can be verified by induction that $(1-\chi_h)/a \in G_h^0$.
By Theorem \ref{calculuscomp0} we have
$$
d = 1 -\chi_h - ((1-\chi_h)/a) \wpr a \in h^2 G_h^{-2}.
$$
Define for $j \in \no$
\begin{align*}
c_j^w(x,D) & = d^w(x,D)^j ((1-\chi_h)/a)^w(x,D) \\
& = ( \underbrace{d \wpr \cdots \wpr d}_{\mbox{\tiny $j$ factors}} \wpr (1-\chi_h)/a )^w(x,D).
\end{align*}
Again by Theorem \ref{calculuscomp0} we have $c_j \in h^{2 j} G_h^{-2j}$.
Set
\begin{equation*}
\wt c_j(z;h) = c_j(z;h)(1-\chi(hz) ), \quad j \in \no.
\end{equation*}
Then $\wt c_j(z/h;h)=0$ for $|z| \leq R$, $h \in (0,1]$ and $j \in \no$, and $\wt c_j \in h^{2 j} G_h^{-2j}$.
According to Lemma \ref{asymptotic2} there exists $c \in G_h^0$ such that
\begin{equation*}
c \sim \sum_0^\infty \wt c_j.
\end{equation*}

Using
$$
((1-\chi_h)/a)^w(x,D) a^w(x,D) = I - d^w(x,D) - \chi_h^w(x,D)
$$
we have
for any integer $k \geq 1$
\begin{align*}
& \sum_{j=0}^{k-1} \wt c_j^w(x,D) a^w(x,D) \\
& = \sum_{j=0}^{k-1} c_j^w(x,D) a^w(x,D) - (c_j \chi_h)^w(x,D) a^w(x,D) \\
& = \sum_{j=0}^{k-1} d^w(x,D)^j (I-d^w(x,D) - \chi_h^w(x,D) ) - (c_j \chi_h)^w(x,D) a^w(x,D) \\
& = I-d^w(x,D)^k - \sum_{j=0}^{k-1} d^w(x,D)^j  \chi_h^w(x,D) - \sum_{j=0}^{k-1}(c_j \chi_h)^w(x,D) a^w(x,D).
\end{align*}
Hence
\begin{equation}\label{regularizing0}
\begin{aligned}
& b_h^w(x,D) c^w(x,D) a^w(x,D) - b_h^w(x,D) \\
& =b_h^w(x,D)   \left(  c^w(x,D) - \sum_{j=0}^{k-1} \wt c_j^w(x,D) \right)a^w(x,D) - b_h^w(x,D) d^w(x,D)^k \\
& \quad - \sum_{j=0}^{k-1} b_h^w(x,D) d^w(x,D)^j  \chi_h^w(x,D) \\
& \quad - \sum_{j=0}^{k-1} b_h^w(x,D)  (c_j \chi_h)^w(x,D) a^w(x,D).
\end{aligned}
\end{equation}

By Definition \ref{asymptoticexpansion1}
we have $c - \sum_{j=0}^{k-1} \wt c_j \in h^{2k} G_h^{-2k}$, and thus
by Theorem \ref{calculuscomp0}
\begin{equation}\label{regularizing1}
b_h \wpr ( c - \sum_{j=0}^{k-1} \wt c_j ) \wpr a \in h^{2k} G_h^{-2k}.
\end{equation}
Moreover, by Theorem \ref{calculuscomp0} we have $d \wpr \cdots \wpr d \in h^{2k} G_h^{-2k}$ ($k$ factors), which gives
\begin{equation}\label{regularizing2}
b_h \wpr  \underbrace{d \wpr \cdots \wpr d}_{\mbox{\tiny $k$ factors}} \in h^{2k} G_h^{-2k}.
\end{equation}

Finally, due to
$$
\supp (b_h) \cap \supp(\chi_h) = \emptyset,
$$
again by Theorem \ref{calculuscomp0} we have for any $0 \leq j \leq k-1$
\begin{equation}\label{regularizing3}
b_h \wpr  \underbrace{d \wpr \cdots \wpr d}_{\mbox{\tiny $j$ factors}}  \wpr \chi_h \in h^{2k} G_h^{-2k}, \qquad
b_h \wpr (c_j \chi_h) \wpr a \in h^{2k} G_h^{-2k}.
\end{equation}
It now follows from \eqref{regularizing1}, \eqref{regularizing2}, \eqref{regularizing3} and \eqref{regularizing0} that we have
\begin{equation*}
b_h \wpr c \wpr  a - b_h \in h^{2k} G_h^{-2k}
\end{equation*}
for $k \in \no$ arbitrary.
\end{proof}

We are now in a position to prove that the test function symbol in the definition of the homogeneous wave front set may be chosen freely as long as the support is sufficiently small.

\begin{thm}\label{HWFinvariance}
Let $u \in \cS'(\rr d)$ and suppose $a \in C_c^\infty(\rr {2d})$,
$z_0 \in \rr {2d} \setminus \{0 \}$, $a(z_0)=1$ and
\begin{equation}\label{hwfassumption}
\| a_h^w(x,D) u \|_{L^2} = \cO(h^\infty), \quad h \in (0,1].
\end{equation}
Then there exists a relatively compact neighborhood $U$ of $z_0$, such that for any $b \in C_c^\infty(\rr {2d})$ with $\supp(b) \subseteq U$ we have
\begin{equation*}
\| b_h^w(x,D) u \|_{L^2} = \cO(h^\infty), \quad h \in (0,1].
\end{equation*}
\end{thm}

\begin{proof}
There exists two relatively compact open neighborhoods $U , U' \subseteq \rr {2d}$ such that $z_0 \in U \subseteq \overline{U} \subseteq U'$,
$0 \notin \overline{U'}$, and
\begin{align*}
|a(z)-1| \leq 1/3, & \quad z \in U, \\
|a(z)-1| \leq 2/3, & \quad z \in U'.
\end{align*}
Let $\chi \in C_c^\infty(\rr {2d})$ satisfy $0 \leq \chi \leq 1$, $\supp(\chi) \subseteq U'$ and $\chi(z)=1$ for $z \in U$, and set
\begin{equation*}
\wt a(z) = a(z) + 2 \| a \|_{L^\infty}(1-\chi(z)) \in G^0.
\end{equation*}
Then $\wt a(z) = a(z)$ when $z \in U$.
For $z \in \rr {2d} \setminus U'$ we have
\begin{equation*}
\left| \wt a(z) \right| = \left| a(z) + 2 \| a \|_{L^\infty} \right|
\geq 2 \| a \|_{L^\infty} - |a(z)| \geq \| a \|_{L^\infty} \geq 1,
\end{equation*}
and for $z \in U'$ we have
\begin{align*}
\left| \wt a(z) \right| & = \left| a(z) -1 + 2 \| a \|_{L^\infty}(1-\chi(z)) + 1 \right| \\
& \geq 2 \| a \|_{L^\infty}(1-\chi(z)) + 1 - |a(z)-1| \\
& \geq 1-2/3=1/3.
\end{align*}
By Lemma \ref{parametrix1} there exists $c \in G_h^0$ and $r \in \bigcap_{k \in \no} h^k G_h^{-k}$ such that
\begin{equation}\label{parametrix2}
\begin{aligned}
b_h^w(x,D) u & = b_h^w(x,D) \, c^w(x,D) \, \wt a_h^w(x,D) u + r^w(x,D)u \\
& = b_h^w(x,D) \, c^w(x,D) \, a_h^w(x,D) u \\
& \quad + b_h^w(x,D) \, c^w(x,D) \, (\wt a-a)_h^w(x,D) u + r^w(x,D) u.
\end{aligned}
\end{equation}

The assumption \eqref{hwfassumption} and the Calder\'on--Vaillancourt theorem now give
\begin{equation*}
\| b_h^w(x,D) c^w(x,D) a_h^w(x,D) u \|_{L^2} = \cO(h^\infty),
\end{equation*}
and, by Proposition \ref{regularize1}, $\| r^w(x,D) u \|_{L^2} = \cO(h^\infty)$.
Due to
$$
\supp (b) \cap \supp(\wt a-a) = \emptyset
$$
it finally follows from Theorem \ref{calculuscomp0} that we have
\begin{equation*}
b_h \wpr c \wpr (\wt a-a)_h \in \bigcap_{k \in \no} h^k G_h^{-k},
\end{equation*}
and therefore by Proposition \ref{regularize1}
\begin{equation*}
\| b_h^w(x,D) c^w(x,D) (\wt a-a)_h^w(x,D) u \|_{L^2} = \cO(h^\infty).
\end{equation*}
From \eqref{parametrix2} we may thus conclude $\| b_h^w(x,D) u \|_{L^2} = \cO(h^\infty)$.
\end{proof}

Finally we deduce a result that is needed in the proof of Theorem \ref{WFequality}.

\begin{prop}\label{kohnnirenberg}
Let $u \in \cS'(\rr d)$, $b \in C_c^\infty(\rr {2d})$,
$z_0 \in \rr {2d} \setminus \{0 \}$, $b(z_0)=1$,
and suppose
\begin{equation}\label{weylassump1}
\| b_h^w (x,D) u \|_{L^2} = \cO(h^\infty), \quad h \in (0,1].
\end{equation}
Then there exists $a \in C_c^\infty(\rr {2d})$ supported in a neighborhood of $z_0$ and equal to one in a smaller neighborhood of $z_0$, such that
\begin{equation*}
\| a_h(x,D) u \|_{L^2} = \cO(h^\infty), \quad h \in (0,1].
\end{equation*}
\end{prop}

\begin{proof}
By Theorem \ref{HWFinvariance} there exists a relatively compact neighborhood $U \subseteq \rr {2d}$ containing $z_0$ such that \eqref{weylassump1} holds for any test function symbol supported in $U$. Let $a \in C_c^\infty(U)$ equal one in a neighborhood of $z_0$
and define an $h$-dependent distribution $c \in \cS'(\rr {2d})$ by $c^w(x,D)=a_h(x,D)$.
By Theorem \ref{KNweyl} we have $c \in G_h^{-\infty}$ and
\begin{equation*}
c (x,\xi) \sim \sum_{j=0}^\infty h^{2j} \frac{\left( -\frac{i}{2} D_x \cdot D_\xi \right)^j}{j!} \, a\big|_{(hx,h\xi)}.
\end{equation*}

Denoting
\begin{equation*}
a_j = (j!)^{-1} \left( - \frac{i}{2} D_x \cdot D_\xi \right)^j a
\end{equation*}
we thus have for any integer $N \geq 1$
\begin{equation*}
c_N := c - \sum_{j=0}^{N-1} h^{2j} (a_j)_h \in h^{2N} G_h^{-\infty}.
\end{equation*}
Note that $\supp(a_j) \subseteq U$ and therefore
\begin{equation}\label{ajhwf}
\| (a_j)_h^w (x,D) u \|_{L^2} = \cO(h^\infty), \quad h \in (0,1].
\end{equation}

For $\alpha \in \nn {2d}$, the symbol $c_N$ differentiated obeys the estimate, for $|z| \geq 1$,
\begin{align*}
\eabs{z}^{N+|\alpha|} |\pdd z \alpha c_N(z) |
& \leq C_{N,\alpha} h^{2N+|\alpha|} |z|^{N+|\alpha|} \eabs{hz}^{-(2N+|\alpha|)-|\alpha|} \\
& \leq C_{N,\alpha} h^{2N+|\alpha|} |z|^{N+|\alpha|} (h|z|)^{-N-|\alpha|} \\
& \leq C_{N,\alpha} h^N,
\end{align*}
and for $|z| < 1$ we have
\begin{align*}
\eabs{z}^{N+|\alpha|} |\pdd z \alpha c_N(z) |
\leq C_{N,\alpha} h^{2N+|\alpha|} \eabs{hz}^{-|\alpha|}
\leq C_{N,\alpha} h^{N}.
\end{align*}
It follows that $c_N \in h^N G^{-N}$.
As in the proof of Proposition \ref{regularize1} it follows that $\| c_N^w(x,D) u \|_{L^2} \lesssim h^N$ provided $N$ is sufficiently large.
Combining with \eqref{ajhwf} we obtain finally for sufficiently large, arbitrary $N \geq 1$
\begin{align*}
\| a_h(x,D) u
\|_{L^2}
& = \| c^w(x,D) u \|_{L^2} \\
& \leq \| c_N^w(x,D) u \|_{L^2} + \sum_{j=0}^{N-1} h^{2j} \| (a_j)_h^w (x,D) u \|_{L^2} \\
& \lesssim h^N.
\end{align*}
\end{proof}

\section*{acknowledgements}

We are grateful to Profs.~D.~Bahns, S.~Coriasco, L.~Rodino, J.~Toft and I.~Witt for valuable advice and constructive criticism. To Prof.~M.~Sugimoto we are grateful for making us aware of the notion of homogeneous wave front set, thus initiating the project.

This work was supported by the German Research Foundation (DFG) through the Institutional Strategy of the University of G\"ottingen, in particular through the Graduiertenkolleg 1493 and the Courant Research Center ``Higher Order Structures in Mathematics''. The first author is also grateful for the support received by the Studienstiftung des Deutschen Volkes, the German Academic Exchange Service (DAAD) and institutional support by the University of Hannover.


\end{document}